\documentclass[reqno]{amsart}

\usepackage{amsfonts, amssymb, pstricks}
\title[Rigidity of Derivations in the Plane]{Rigidity of 
Derivations in the Plane and in Metric Measure Spaces}
\author{Jasun Gong}
\email{jasun.gong@aalto.fi}
\address{Jasun Gong \hfill\break\indent
Institute of Mathematics, Aalto University \hfill\break\indent 
P.O. Box 11100, FI-00076 Aalto
\hfill\break\indent Finland
}
\date{9 August 2011}

\theoremstyle{plain}
\newtheorem{thm}{Theorem}[section]
\newtheorem{prop}[thm]{Proposition}

\newtheorem{cor}[thm]{Corollary}
\newtheorem{lemma}[thm]{Lemma}
\newtheorem{ques}[thm]{Question}
\newtheorem{conj}[thm]{Conjecture}

\theoremstyle{definition}
\newtheorem{defn}[thm]{Definition}
\newtheorem{rmk}[thm]{Remark}
\newtheorem{eg}[thm]{Example}

\theoremstyle{remark}

\newtheorem{claim}[thm]{Claim}

\numberwithin{equation}{section}

\renewcommand{\a}{\alpha}
\renewcommand{\b}{\beta}

\renewcommand{\d}{\delta}

\newcommand{\diam}{\operatorname{diam}}

\newcommand{\e}{\epsilon}
\renewcommand{\H}{\mathcal{H}}
\newcommand{\id}{\operatorname{id}}

\newcommand{\G}{\mathcal{G}}

\newcommand{\Li}{L^\infty}
\newcommand{\lip}{\operatorname{lip}}
\newcommand{\Lip}{\operatorname{Lip}}
\newcommand{\Lipb}{\operatorname{Lip}_b}
\newcommand{\m}{\mathcal{M}}
\newcommand{\n}{\mathcal{N}}
\newcommand{\N}{\mathbb{N}}
\newcommand{\proj}{\operatorname{proj}}
\newcommand{\R}{\mathbb{R}}

\newcommand{\spt}{\operatorname{spt}}
\newcommand{\U}{\Upsilon}
\newcommand{\Z}{\mathbb{Z}}

\newcommand{\wkto}{\rightharpoonup}
\newcommand{\wsto}{\stackrel{*}{\rightharpoonup}}

\def\Xint#1{\mathchoice
   {\XXint\displaystyle\textstyle{#1}}%
   {\XXint\textstyle\scriptstyle{#1}}%
   {\XXint\scriptstyle\scriptscriptstyle{#1}}%
   {\XXint\scriptscriptstyle\scriptscriptstyle{#1}}%
   \!\int}
\def\XXint#1#2#3{{\setbox0=\hbox{$#1{#2#3}{\int}$}
     \vcenter{\hbox{$#2#3$}}\kern-.5\wd0}}

\def\dashint{\Xint-}

\begin{document}

\begin{abstract}
Following Weaver \cite{WeaverED} we study generalized differential 
operators, called (metric) derivations, and their linear algebraic 
properties.
In particular, for $k = 1, 2$ we show that measures on $\R^k$ that induce 
rank-$k$ modules of derivations must be absolutely continuous to Lebesgue 
measure.  An analogous result holds true for measures concentrated on 
$k$-rectifiable sets with respect to $k$-dimensional Hausdorff measure.

Though formulated for Euclidean spaces, these rigidity results also apply 
%to the general setting of metric spaces 
to the metric space setting and specifically, to spaces that support a
doubling measure and a $p$-Poincar\'e inequality.  Using our results for
the Euclidean plane, we prove the $2$-dimensional case of a conjecture of
Cheeger, which concerns the non-degeneracy of Lipschitz images of such
spaces.
\end{abstract}

\maketitle

%=========================================================================
\tableofcontents

\section{Introduction} \label{sect_intro}

In this work we consider Weaver's theory of (metric) derivations 
\cite{WeaverED}, which are generalizations of differential operators 
on Riemannian manifolds.  For metric spaces equipped with a Borel measure, 
derivations are linear operators from the class of bounded Lipschitz 
functions to the class of essentially bounded functions with respect to 
certain weak topologies; see Lemma \ref{lemma_weakstarlip} and Definition 
\ref{defn_deriv}.

On $\R^n$ equipped with the standard metric, Rademacher's theorem states that 
every Lipschitz function is almost everywhere (a.e.) differentiable with 
respect to the Lebesgue measure.  Put one way, the validity of 
Rademacher's theorem is encapsulated in the structure of a metric space,
if there exists a nonzero derivation with respect to a fixed measure on 
that space.

\subsection{Rigidity of Measures and Derivations} \label{sect_rigidity}

The framework of \cite{WeaverED} includes many examples, such as 
Riemannian manifolds, the self-similar fractal spaces of Laakso 
\cite{Laakso}, and infinite-dimensional spaces such as Banach manifolds 
and abstract Wiener spaces.  In each example, there are natural choices 
for the metric and measure, but one may inquire as to how flexible these 
choices can be made.

\begin{ques} \label{ques_howmany}
On a given metric space, which measures induce nontrivial derivations?
Of those, how many can we expect?
\end{ques}

For a fixed space, the set of derivations admits a natural module 
structure, so the notions of {\em linear independence} and {\em basis} 
are well-defined for derivations.  We therefore determine ``how many'' 
derivations exist on a space in terms of the \emph{rank} of the module.

To clarify, Question \ref{ques_howmany} is not simply a matter of the
Hausdorff dimensions of the relevant spaces, but of subtler issues of
geometry as well.  Given a line in $\R^n$, for instance, $1$-dimensional
Hausdorff measure induces a rank-$1$ module of derivations
\cite[Thm 38]{WeaverED}.  On the other hand, the ``middle-thirds''
Sierpi\'nski carpet in $\R^2$ equipped with its natural Hausdorff measure
(of dimension $\log_38$) does not admit {\em any} nonzero derivations
\cite[Thm 41]{WeaverED}.

In this paper we will focus on the case of Euclidean spaces.  The following 
result indicates that, for $k = 1,2$, there are few choices of Radon 
measures on $(\R^k, | \cdot |)$ that induce rank-$k$ modules of 
derivations.  

\begin{thm} \label{thm_rigiditylowdim} 
Let $k \in \{1,2\}$.  If $\mu$ is a Radon measure on $(\R^k, | \cdot |)$ 
that induces a rank-$k$ module of derivations, then it is  absolutely 
continuous to %$k$-dimensional
Lebesgue measure.  Moreover, derivations with respect to $\mu$ are linear
combinations of the differential operators $\{\partial / \partial x_i\}_{i=1}^k$
with scalars in $L^\infty(\R^k,\mu)$.
\end{thm}

The class of Lipschitz functions on a space clearly depends on the choice
of metric on that space.  So in terms of derivations, Theorem
\ref{thm_rigiditylowdim} can be viewed as a {\em rigidity} result for
measures on $\R^k$ that obey a Rademacher-type property.

Regarding the $k=2$ case, the proof uses a recent result of Alberti, 
Cs\"ornyei, and Preiss \cite{ACP} about the structure of Lebesgue null 
sets in $\R^2$.  Roughly speaking, it asserts that every Lebesgue null
set in $\R^2$ (that is, a subset of zero Lebesgue measure) splits
into a horizontal part and a vertical part.  So given a measure $\mu$
that is concentrated on such a set, we show that each part admits a
generalized``tangent'' vector field whose components satisfy a linear
dependence relation for all derivations with respect to $\mu$; see Lemma
\ref{lemma_linrel}.

The remaining case of rank-$1$ modules on $\R^2$ is not well-understood.  
In this direction, S.\ Wenger has observed that \emph{on a complete, 
separable metric space, every $1$-dimensional current in the sense of 
Ambrosio and Kirchheim \cite{AK} %\cite[Defn 3.1]{AK}
determines a derivation, where the underlying measure is the mass of the
current}.  Conversely it is easily shown that every derivation induces a
$1$-dimensional current.  The problem of classifying rank-$1$ modules of
derivations on $\R^k$ is therefore equivalent to the so-called ``Flat
Chain Conjecture'' about $1$-dimensional currents \cite[Sect 11]{AK}. 
For more about currents on metric spaces, see \cite{AK},
\cite{HardtdePauw}, and \cite{Lang}.

\subsection{Applications to Metric Spaces}

Though formulated for Euclidean spaces, the results in \S\ref{sect_rigidity}
are also surprisingly relevant to the general setting of {\em metric
measure spaces} --- that is, metric spaces equipped with Borel measures.

To obtain a reasonable setting for analysis, we restrict our focus to
spaces that support doubling measures.  Recall that a Borel measure
$\mu$ on $(X,d)$ is called {\em doubling} if there exists $\kappa \geq 1$
so that
\begin{equation} \label{eq_doubling}
0 \;<\; \mu(B(x,2r)) \;\leq\; \kappa \, \mu(B(x,r)) \;<\; \infty
\end{equation}
holds for all $x \in X$ and all $r > 0$.  Spaces supporting such measures
are particular cases of {\em spaces of homogeneous type} \cite{CoifmanWeiss};
in particular they have finite Hausdorff dimension and admit generalized
dyadic-cube decompositions \cite{Christ}.  Intuitively, the doubling
condition \eqref{eq_doubling} ensures that the space $X$ has good
scaling properties, from which we obtain a rich theory of ``zeroth order''
calculus --- that is, good analogues of Riesz potentials, the Lebesgue
differentiation theorem, and other elements of harmonic analysis.
%that includes the celebrated $H^1$-$BMO$ duality of Fefferman and Stein \cite{FeffermanStein}.

For a theory of first-order calculus, however, we also require the spaces
to support a generalized Poincar\'e inequality.  Indeed, on $\R^n$ the
inequality takes the form
$$
\dashint_{B(x,r)} |f - f_{B(x,r)}| \, dx \;\leq\;
C(n,p) \, r \, \Big( \dashint_{B(x,r)} |\nabla f|^p \,dx \Big)^{1/p}
$$
for all $p \geq 1$ and all Lipschitz $f :\R^n \to \R$.  Here
mean values are denoted by
$$
f_A \;:=\; \dashint_A f \,dx \;:=\; \frac{1}{|A|}\int_A f \,dx.
$$
So at sufficiently small scales, the inequality guarantees that 
``discrete gradients'' $(f - f_{B(x,r)})/r$ are comparable to the usual
gradients $|\nabla f|$ in an averaged sense.

There is an analogous formulation of the Poincar\'e inequality for
metric spaces supporting doubling measures.  In this setting, upper
gradients replace the usual gradients, but there remains the same
consequence that Lipschitz functions have good infinitesmal behavior
\cite{HeinonenKoskela}, \cite{Sh}. 
Indeed, Cheeger \cite{Cheeger} has shown that Lipschitz functions on
such spaces are also a.e.\ differentiable; see Theorem
\ref{thm_cheegerrademacher}.  As a result, these spaces admit generalized
differentiable structures, and Keith has extended the result for a more
general class of spaces \cite{Keith}.

In particular, Theorem \ref{thm_rigiditylowdim} gives rise to {\em
geometric rigidity} theorems in cases when the metric space embeds
isometrically into a Euclidean space.  For example, it implies an
affirmative answer to a conjecture of Cheeger \cite[Conj 4.63]{Cheeger}
when the generalized differentiable structure is $2$-dimensional ---
that is, the $N \leq 2$ case of Theorem \ref{thm_cheegerrademacher}.
The statement of the conjecture is technical, but combined with \cite[Thm
14.2]{Cheeger} it implies the following result.

\begin{thm} \label{thm_bilipembed}
Let $(X,d)$ be a complete metric space that supports a doubling measure 
$\mu$ and a $p$-Poincar\'e inequality.  If the corresponding 
measurable differentiable structure is (at most) $2$-dimensional and if 
there is an isometric embedding $\iota : X \to \R^N$, for some $N \in \N$, 
then the image $\iota(X^m)$ of each coordinate chart $X^m$ is an
$n(m)$-rectifiable set.
\end{thm}

In the context of geometric measure theory, it is a fact that every
$n$-rectifiable set in $\R^N$ agrees with a countable union of
$n$-dimensional, $C^1$-smooth submanifolds, up to a set of zero
$n$-dimensional Hausdorff measure  \cite[Thm 15.21]{Mattila}.  Theorem
\ref{thm_bilipembed} therefore asserts that $2$-dimensional spaces
supporting Euclidean metrics and nontrivial derivations must also have
locally Euclidean geometry (up to negligible
subsets).

As a special case, Theorem \ref{thm_rigiditylowdim} implies another 
rigidity theorem for measures in the plane.  The case of $\R$ was proven 
by Bj\"orn, Buckley, and Keith \cite{BBK}.

\begin{thm} \label{thm_euclabsocty}
Let $\mu$ be a doubling measure on $\R^2$ whose support is dense in 
$\R^2$.  If $(\R^2, |\cdot |, \mu)$ supports a $p$-Poincar\'e 
inequality, then $\mu$ is absolutely continuous to Lebesgue measure.
\end{thm}

The hypotheses of a Poincar\'e inequality and the density of the support
of $\mu$ are necessary for the theorem.  Namely, there exist doubling
measures on $\R^n$ that are singular to Lebesgue measure; for examples,
see \cite{KaufmanWu}, \cite{Wu}, and \cite{GKS}.  Moreover, certain
non-self-similar Sierpi\'nski carpets in $\R^2$ support both a doubling
measure (as restricted to the carpet) and a $p$-Poincar\'e inequality
\cite{MackayTysonWildrick}. However, such measures are {\em porous} over
all of $\R^2$, so Theorem \ref{thm_euclabsocty} does not apply.

We note that Keith has proven \cite[Conj 4.63]{Cheeger} for 
$1$-dimensional differentiable structures, and that our methods are
independent of his.  His proof relies on a fact about sets of
non-differentiability of Lipschitz functions on $\R$ \cite{PreissTiser}.
Alberti, Cs\"ornyei, and Preiss have recently announced an analogous
fact in $\R^2$ \cite[Thm 7.5]{ACP} and from this, Keith's techniques
will also prove the $2$-dimensional case of Cheeger's conjecture.

%==================================
\subsection{Plan of the Paper \& Acknowledgments}

Section 2 begins by introducing terminology and recalling basic facts
about Lipschitz functions.  It also contains the basics of Weaver's
theory, clarifies the equivalence of definitions from \cite{WeaverED}
and \cite{HeinonenNC}, and gives new facts about derivations on metric
measure spaces.  

The case of derivations on $1$-dimensional sets in $\R^n$ is treated in 
Section 3; this includes the setting of $1$-rectifiable sets.  In Section 
4 we discuss the structure of Lebesgue null sets in $\R^2$ and the rigidity
of measures that induce rank-$2$ modules of derivations.  Section 5 begins 
with basic facts about spaces admitting a Poincar\'e inequality and 
concludes with a proof of the $2$-dimensional case of Cheeger's conjecture;
we also explore the relationship between several open problems.

The author would like to thank Mario Bonk and Pekka Pankka for many valuable discussions. %about derivations and other related topics.  
He would also like to thank Bruce Kleiner and Stefan Wenger for their useful comments on a preliminary report of this work.

The author is particularly indebted to to his late advisor, Juha Heinonen, for suggesting this direction of research and more generally, for all of the help and guidance over the years.

The author was supported by NSF RTG grant \#0602191.

%=========================================================================
\section{Preliminaries} \label{sect_basics}

\subsection{Notation and Preliminaries}

The identity map on a set $S$ is denoted by $\id_S$.  For real-valued functions $f$ and $g$, we denote their pointwise minimum and maximum as $f \wedge g$ and $f \vee g$, respectively.

For a measure $\mu$ on a set $X$ and a $\mu$-measurable subset $A$ of 
$X$, the restriction measure $\mu \lfloor A$ is defined as 
$$
(\mu\lfloor A)(E) \;:=\; \mu(A \cap E)
$$ 
for all $\mu$-measurable subsets $E$ in $X$. %$\mu$-measurable subsets $E$ of $X$.
If $\mu = \mu \lfloor A$, then we say that $\mu$ is \emph{concentrated} on $A$.
A collection $\{X_i\}_{i=1}^\infty$ of $\mu$-measurable subsets of $X$ is %said to be 
a \emph{$\mu$-measurable decomposition} of $X$ if $\mu$ is concentrated on $\bigcup_{i=1}^\infty X_i$ and if $\mu(X_i \cap X_j) = 0$ holds whenever $i \neq j$.

Given $p \in [1,\infty]$ and a measure $\mu$ on a set $X$, the standard 
norm on the Banach space $L^p(X,\mu)$ is denoted by $\| \cdot \|_{\mu,p}$.  
We will write $\|f\|_\infty$ for the supremum norm of a function $f$, 
whenever it exists.  As indicated before, given a function
$u \in L^1_{loc}(X,\mu)$ and a subset $A \subset X$ with
$0 < \mu(A) < \infty$, its mean value is
$$
u_A \;:=\; 
\dashint_A u \,d\mu \;=\; 
\frac{1}{\mu(A)} \int_A u \,d\mu.
$$
On a metric space $X$, a measure $\mu$ is Radon if it is Borel regular
and if balls have positive finite $\mu$-measure.  We will denote
$\a$-dimensional Hausdorff measure on a metric space $X$ by $\H^\a_X$.
For $X = \R^n$, we write $\H^\a = \H^\a_X$ and $m_n$ for the Lebesgue
measure.

The standard basis of vectors on $\R^n$ is denoted by $\{ e_1, e_2,
\ldots, e_n \}$.  If $V$ is a linear subspace of $\R^n$, then $\proj_V : 
\R^n \to V$ is the orthogonal projection map onto $V$.    
For $j = 1, 2, \ldots, n$, the standard partial differential operators on 
$\R^n$ are denoted by $\partial_j := \partial/\partial x_j$.  The 
class of smooth functions on $\R^n$ with compact support is denoted by
$C^\infty_0(\R^n)$.

%==================================
\subsection{Lipschitz Functions}

Let $(X,\rho_X)$ and $(Y,\rho_Y)$ be metric spaces.  
Recall that a function $f : X \to Y$ is Lipschitz if 
\begin{equation} \label{eq_lipconst}
L(f) \;:=\;
\sup\left\{ \frac{\rho_Y(f(x),f(y))}{\rho_X(x,y)} \,:\, x, y \in X, \, x \neq y \right\} \;<\; \infty
\end{equation}
and we refer to $L(f)$ as the \emph{Lipschitz constant} of $f$.

We write $\Lip(X;Y)$ for the space of Lipschitz maps from $X$ to $Y$ and $\Lip_b(X;Y)$ for the subspace of bounded Lipschitz maps in $\Lip(X;Y)$.  For $Y = \R$, we write 
$$
\Lip(X) \,:=\, \Lip(X;\R) \; \text{ and } \; \Lip_b(X) \,:=\, \Lip_b(X;\R).
$$

We now recall some basic properties of Lipschitz maps.  Their proofs are 
elementary and we omit them.

\begin{lemma} \label{lemma_lipprops}
Let $X$, $Y$, and $Z$ be metric spaces.
\begin{enumerate}
\item
If $f \in \Lip(X;Y)$ and $g \in \Lip(Y;Z)$, then $g \circ f \in \Lip(X;Z)$.
\item
$\Lip(X)$ is a vector space, and $\Lip_b(X)$ is an algebra over $\R$.
\item
If $f$ and $g$ are functions in $\Lip(X)$, then so are $f \vee g$ and $f \wedge g$.
\item
Let $A$ be a closed subset of $X$.  If $f \in \Lip(A)$, then there exists $F \in \Lip(X)$ so that $F|A = f$ and $L(F) = L(f)$.
\end{enumerate}
\end{lemma}

Part (4) of Lemma \ref{lemma_lipprops} is known as the 
\emph{McShane-Whitney extension} of a Lipschitz function \cite{McShane}, 
\cite{Whitney}.  For $f \in \Lip(A)$, an explicit formula is
$$
F(x) \;:=\; \inf\{ f(a) + L(f) \cdot d(x,a) : a \in A \}.
$$
Combining Parts (3) and (4), we obtain an analogous fact for the space
$\Lip_b(X)$.

\begin{lemma} \label{lemma_bddmcshane}
Let $A$ be a closed subset of $X$.  If $f \in \Lip_b(A)$, then there 
exists $f^A \in \Lip_b(X)$ so that $f^A|A = f$, $L(f^A) = L(f)$, and 
$\|f^A\|_\infty = \|f\|_\infty$.
\end{lemma}

We refer to $f^A$ as the \emph{bounded McShane extension} of $f$. Note that 
$\Lip_b(X)$ is a Banach space with respect to the norm
\begin{equation} \label{eq_lipnorm}
\|f\|_{\Lip} \;:=\; \|f\|_\infty \vee L(f).
\end{equation}
For a proof, see \cite[Prop 1.6.2(a)]{WeaverBook}.  In fact, $\Lip_b(X)$
is a {\em dual} Banach space \cite{AE}; see also \cite[Prop 2 \& 8]{WeaverLA}.

\begin{lemma}[Weaver, 1996] \label{lemma_weakstarlip}
Let $X$ be a metric space.
\begin{enumerate}
\item
$\Lip_b(X)$ is a dual Banach space with respect to the norm $\| \cdot \|_{\Lip}$.
\item
If $\{f_\a\}_{\a \in I}$ is a bounded net in $\Lip_b(X)$, then $f_\a$ converges weak-$*$ to $f$ in $\Lip_b(X)$ if and only if $f_\a$ converges pointwise to $f$.
\end{enumerate}
\end{lemma}

The next lemma follows from the first lemma and from several explicit
constructions in \cite[Sect 1.7 \& 2.2]{WeaverBook}.

\begin{lemma} \label{lemma_arenseells}
If $(X,\rho)$ is a separable metric space, then the pre-dual of $\Lip_b(X)$ is a separable Banach space.
\end{lemma}

\begin{proof}%[Proof of Lemma \ref{lemma_arenseells}]
Clearly, $\rho_2 := \rho \wedge 2$ is a metric on $X$ and the metric
space $X_2 := (X,\rho_2)$ is separable because $(X,\rho)$ is separable.
By \cite[Prop 1.7.1]{WeaverBook}, we have the isometric isomorphism
$\Lip_b(X) \cong \Lip_b(X_2)$.

Let $X_2^+$ be the set of all points in $X_2$ as well as one additional
point $e$, so $X_2^+ = X \cup \{e\}$.  We may extend $\rho_2$ to a metric
$\rho_2^+$ on $X_2^+$ by the rule
$$
\rho_2^+(x,e) \;:=\;
\left\{\begin{array}{rl}
\diam(X_2), & x \neq e \\
0, & x = e.
\end{array}\right.
$$
Now consider the space of functions given by
$$
\Lip_0(X_2^+) \;:=\; \{ f \in \Lip(X_2^+) \,:\, f(e) = 0\}.
$$
By \cite[Prop 1.6.2(b)]{WeaverBook}, $\Lip_0(X_2^+)$ is a Banach space
with respect to the norm $f \mapsto L(f)$.  Moreover, by
\cite[Thm 1.7.2]{WeaverBook} we also have the isometric isomorphism
$\Lip_0(X_2^+) \cong \Lip_b(X_2)$.

Clearly $X_2^+$ is bounded, so by \cite[Thm 2.2.2]{WeaverBook} the pre-dual of $\Lip_0(X_2^+)$ is isometrically isomorphic to the \emph{Arens-Eells space} $AE(X_2^+)$.  Since $X_2^+$ is a bounded, separable metric space, it follows by construction \cite[Defn 2.2.1]{WeaverBook} that $AE(X_2^+)$ is a separable Banach space.  This gives the isometric isomorphism
$$
\Lip_b(X) \;\cong\; [AE(X_2^+)]^*. \qedhere
$$
%and proves the lemma.
\end{proof}

%==================================
\subsection{Derivations \& Basic Properties}

Here and in what follows, triples of the form $(X,\rho,\mu)$ will
denote a metric space $(X,\rho)$ equipped with a Borel measure $\mu$.

\begin{defn}[Weaver] \label{defn_deriv}
A \emph{derivation} $\d \colon \Lipb(X) \to L^\infty(X,\mu)$ is a linear map
that satisfies 
%the following two properties:
\begin{enumerate}
\item the Leibniz Rule: $\d (f \cdot g) \,=\, f \cdot \d g + g \cdot \d f$ holds for all $f, g \in \Lipb(X)$;
\item Weak-$*$ continuity on bounded sets: if $\{f_i\}_{i \in I}$ is a bounded net in $\Lipb(X)$ so that $f_i \wsto f$, then $\d f_i \wsto \d f$ in $\Li(X,\mu)$.
\end{enumerate}
\end{defn}

The set of derivations on $(X,d,\mu)$ is denoted by $\U(X,\mu)$.

\begin{rmk} \label{rmk_constzero}
The Leibniz rule implies that $\d(1) = 2 \cdot \d(1)$, so $\d c = 0$ holds
for every constant $c \in \R$.
\end{rmk}

\begin{rmk}
Property (2) in Definition \ref{defn_deriv} is better known as \emph{bounded weak-$*$ continuity}, which refers to continuity with respect to the \emph{bounded weak-$*$ topology} on the space of linear operators between Banach spaces.  For more
about this topology, see \cite[Thm V.5.3]{DunfordSchwartz}.
\end{rmk}

We will refer to Property (2) as the \emph{continuity property of derivations}, or simply as \emph{continuity}.  For separable metric spaces, this property reduces to familiar modes of continuity from functional analysis.  In particular, Definition \ref{defn_deriv} agrees with that in \cite[p.68]{HeinonenNC}.

\begin{lemma} \label{lemma_ctyderiv}
Let $(X,\rho)$ be a separable metric space, let $\mu$ be a measure on $X$, and let $\d : \Lip_b(X) \to \Li(X,\mu)$ be a linear operator.
\begin{enumerate}
\item If $\d \in \U(X,\mu)$, then $\d$ is a bounded operator.
\item $\d \in \U(X,\mu)$ holds if and only if $\d$ satisfies the Leibniz rule and is weak-$*$ continuous with respect to sequences in $\Lip_b(X)$.
\end{enumerate}
\end{lemma}

\begin{proof}
Suppose there is a $\d \in \U(X,\mu)$ with the property that, for each $n \in \N$, there exists $f_n \in \Lip_b(X)$ so that $\|f_n\|_{\Lip} \leq 1$ and $\|\d f_n\|_{\infty,\,\mu} \geq n$.

By Lemma \ref{lemma_arenseells}, $\Lip_b(X)$ is the dual of a separable Banach space, so by the Banach-Alaoglu Theorem \cite[Thm 3.17]{RudinFA}, it follows that there is a weak-$*$ convergent subsequence $\{f_{n_m}\}_{m=1}^\infty$.  The sequence $\{\d f_{n_m}\}_{m=1}^\infty$ is weak-$*$ convergent in $\Li(X,\mu)$, by the continuity property of derivations, and therefore bounded.  On the other hand, we have, by hypothesis,
$$
\|\d f_{n_m}\|_{\infty,\,\mu} \;\geq\; n_m \;\to\; \infty.
$$
as $m \to \infty$.  This is a contradiction, which gives Part (1).

Since $\Lip_b(X)$ has a separable pre-dual, by \cite[Thm 3.16]{RudinFA} the weak-$*$ topology on $\bar{B}(0,R) \subset \Lip_b(X)$ is metrizable.  As a result, weak-$*$ convergence on bounded sets in $\Lip_b(X)$ agrees with weak-$*$ convergence with respect to sequences; this gives Part (2).
\end{proof}

By Lemma \ref{lemma_ctyderiv}, each derivation in $\U(X,\mu)$ has a well-defined operator norm whenever $X$ is separable.  We will denote this norm by
\begin{equation} \label{eq_opnorm}
\|\d\|_{\rm op} \;:=\;
\sup\big\{ \|\d f\|_{\infty,\,\mu} \,:\,
f \in \Lip_b(X),\, \|f\|_{\Lip} \leq 1 \big\}.
\end{equation}

We now give two examples of metric measure spaces and their derivations.

\begin{eg}\cite[Sect 5B]{WeaverED}  \label{eg_eucl}
For $i = 1, 2, \ldots, n$, each $\partial_i$ lies in $\U(\R^n,m_n)$.
The continuity follows from an integration by parts argument.
\end{eg}

\begin{eg}\cite[Cor 35]{WeaverED} \label{eg_cantor}
If $\mu$ is any measure on $\R$ that is concentrated on the `middle-thirds' Cantor set, then $\U(\R,\mu) = 0$.
This fact follows from Lemma \ref{lemma_dim1}, but the original proof in \cite{WeaverED} relies on the \emph{total disconnectedness} and self-similarity of the Cantor set.
\end{eg}

As stated in the Introduction, the set $\U(X,\mu)$ is a module over the ring $\Li(X,\mu)$.  Indeed, for $\d \in \U(X,\mu)$ and $\lambda \in \Li(X,\mu)$, we define $\lambda \cdot \d \in \U(X,\mu)$ by the rule
$$
(\lambda \cdot \d) f(x) \;:=\; \lambda(x) \cdot \d f (x).
$$

\begin{defn} \label{defn_LI}
A set $\{\d_i\}_{i=1}^M$ \emph{generates} $\U(X,\mu)$ if, for all $\d \in \U(X,\mu)$, there are scalars $\{c_i\}_{i=1}^M$ in $\Li(X,\mu)$ so that $\d = \sum_{i=1}^M c_i \d_i$.

A set $\{\eta_i\}_{i=1}^N$ is \emph{linearly independent} in $\U(X,\mu)$ if whenever there are scalars $\{\lambda_i\}_{i=1}^N$ in $\Li(X,\mu)$ so that $\sum_i \lambda_i \eta_i = 0$, then each $\lambda_i$ is zero.
The \emph{rank} of the module $\U(X,\mu)$ is the largest cardinality of any linearly independent set in $\U(X,\mu)$.
\end{defn}

The next lemma follows directly from Definition \ref{defn_LI}, so we omit the proof.

\begin{lemma} \label{lemma_lindepsubset}
Let $N \in \N$ and let $A$ be a $\mu$-measurable subset of $X$ with $\mu(A) > 0$.  If $\{\d_i\}_{i=1}^N$ is a linearly independent set in $\U(X,\mu)$, then $\{\chi_A\d_i\}_{i=1}^N$ is also a linearly independent set in $\U(X,\mu)$.
\end{lemma}

\begin{eg}\cite[Thm 37]{WeaverED} \label{eg_riem}
For $X = \R^n$, $\{\partial_i\}_{i=1}^n$ is a linearly independent set that generates $\U(\R^n,m_n)$. Moreover, as $\Li(\R^n,m_n)$-modules, %we have the module isomorphism
$$
\U(\R^n,m_n) \;\cong\; \bigoplus_{i=1}^n \Li(\R^n,m_n).
$$
More generally, let $X$ be a compact Riemannian manifold and let $\mu$ be 
the volume element. Then $\U(X,\mu)$ is isomorphic to the 
$\Li(X,\mu)$-module of bounded measurable sections of the tangent bundle $TX$.
\end{eg}

Derivations are also known as \emph{measurable vector fields} in 
\cite{HeinonenNC} and \cite{WeaverED}.  In the remainder of the section, 
we investigate properties of derivations which are similar to those of 
vector fields on smooth manifolds.

%==================================
\subsection{Locality \& Applications}

On a smooth manifold $M$, vector fields are \emph{local} objects; that 
is, their action on a function $\varphi \in C^\infty(M)$ near a point $x 
\in M$ depends only on the behavior of $\varphi$ near $x$.  The next 
theorem shows that derivations enjoy a similar property, called the {\em 
locality property}.  It is a special case of \cite[Thm 29]{WeaverED}; see 
also \cite[Thm 13.3]{HeinonenNC} and \cite[Sect 3.2]{Gong}.

\begin{thm}[Weaver, 2000] \label{thm_locality}
Let $\mu$ be a Radon measure on $X$.  If $A$ is a $\mu$-measurable subset of $X$, then we have the
$\Li(X,\mu)$-module isomorphism
$$%\begin{eqnarray*} \label{eq_locality}
\U(A,\mu) \;\cong\;
\chi_A \cdot \U(X,\mu) \;:=\;
\{ \chi_A\d \,:\, \d \in \U(X,\mu) \}.
$$%\end{eqnarray*}
\end{thm}

By definition, each $\d \in \U(X,\mu)$ acts only on bounded Lipschitz functions.  In the case of Radon measures $\mu$, however, the domain of definition of $\d$ extends to include all Lipschitz functions.

\begin{thm} \label{thm_extlip}
Let $\mu$ be a Radon measure on $X$.  For each $\d \in \U(X,\mu)$, there 
is a linear map $\bar\d : \Lip_{loc}(X) \to L^\infty_{loc}(X,\mu)$ with 
the following properties:
\begin{enumerate}
\item 
for all $f \in \Lip_b(X)$, we have $\bar\d f = \d f$;
\item 
for all $f \in \Lip(X)$ and all balls $B$ in $X$, we have
$\chi_B \bar\d f = \chi_B \d\big((f|B)^B\big)$;
\item the Leibniz rule holds for $\bar\d$;
\item if $X$ is separable, then for all $f \in \Lip(X)$, we have
$\|\bar\d f\|_{\infty,\,\mu} \leq \|\d\|_{\rm op} L(f).$
\end{enumerate}
\end{thm}

To reiterate, $(f|B)^B$ refers to the bounded McShane extension of $f|B$.  Theorem \ref{thm_extlip} will follow from the next lemma and a locality argument.

\begin{lemma} \label{lemma_gluing}
Let $X$ be a separable metric space, let $\mu$ be a Radon measure on $X$, and let $\{X_i\}_{i=1}^\infty$ be a $\mu$-measurable decomposition of $X$.  Suppose that for each $i \in \N$, there exists $\d_i \in \U(X_i,\mu)$ so that $\|\d_i\|_{\rm op} \leq 1$.  Then the linear operator $\d : \Lip_b(X) \to \Li(X,\mu)$, given by
$$
\d f \;:=\; \sum_{i=1}^\infty \chi_{X_i} \cdot \d_i(f|X_i)
$$
determines a derivation in $\U(X,\mu)$, with $\|\d\|_{\rm op} \leq 1$.
\end{lemma}

\begin{proof}[Proof of Lemma \ref{lemma_gluing}]
%Note that $\d$ is well-defined, because 
For each $f \in \Lip_b(X)$, we have
$$
\mu(\{ x \,:\, |\d f(x)| > 1 \}) \;\leq\;
\sum_{i=1}^\infty \mu(\{ x \in X_i \,:\, |\d_i(f|X_i)(x)| > 1 \}) \;=\; 0.
$$
Therefore $\d$ is well-defined and $\|\d f\|_{\mu,\,\infty} \leq 1$.  The map $\d$ is clearly linear and satisfies the Leibniz rule.  To check continuity, let $\{f_\a\}_{\a \in I}$ be a net in $\Lip_b(X)$ so that $f_\a \to 0$ and $\sup_\a \|f_\a\|_{\Lip} \leq L$ holds, for some $L \in [0,\infty)$.

Let $\e > 0$ be arbitrary.  For each $h \in L^1(X,\mu)$, there is an $N \in \N$ so that
$$
\sum_{i=N+1}^\infty \int_{X_i} |h| \,d\mu \;\leq\; \frac{\e}{2L \cdot \|\d\|_{\rm op}}.
$$
Moreover, for each $i = 1, 2, \ldots, N$, we have $h|X_i \in L^1(X_i,\mu)$, so the bound
$$
\left| \int_{X_i} h \cdot \d f_\a \,d\mu \right| \;=\;
\left| \int_{X_i} h \cdot \d_i(f|X_i) \,d\mu \right| \;<\; 
\frac{\e}{2N}
$$
follows from the continuity of $\d_i$.  We then compute
\begin{eqnarray*}
\left| \int_X h \cdot \d f_\a \,d\mu \right| &\leq&
\left| \sum_{i=1}^N \int_{X_i} h \cdot \d f_\a \,d\mu \right| \,+\,
\sum_{i=N+1}^\infty \|\d f_\a\|_{\mu,\, \infty} \cdot \int_{X_i} |h|
\,d\mu \\ &\leq&
N \cdot \frac{\e}{2N}
\,+\, L \cdot \|\d\|_{\rm op} \cdot \frac{\e}{2L \cdot \|\d\|_{\rm op}} \;=\; \e.
\end{eqnarray*}
As a result, we have $\d f_\a \wsto 0$ in $\Li(X,\mu)$, which proves the lemma.
\end{proof}

\begin{proof}[Proof of Theorem \ref{thm_extlip}]
Without loss of generality, let $\{X_n\}_{n=1}^\infty$ be a $\mu$-measurable decomposition of $X$ so that each $X_n$ is a bounded set.  (For example, fix a base point $a \in X$ and put $X_n := B(a,n) \setminus B(a,n-1)$ for $n \in \N$.)  
Put
$$
\bar\d f \;:=\; \sum_{n=1}^\infty \chi_{X_n} \cdot \d((f | X_n)^{X_n}).
$$
Indeed, $\bar\d f$ is well-defined because $f|X_n \in \Lip_b(X_n)$ holds, for all $n \in \N$.  Clearly $\bar\d$ is linear.  By Theorem \ref{thm_locality}, we have
$$
\chi_{X_n} \cdot \d((f | X_n))^{X_n} \;=\;
\chi_{X_n} \cdot \d f
$$
for all $f \in \Lip_b(X)$ and all $n \in \N$, so Property (1) follows.  Similarly, for each $n \in \N$ and each ball $B$ in $X$, Property (2) follows from
$$
\chi_{B \cap X_n} \cdot \d_{X_n}\d((f | X_n))^{X_n} \;=\;
\chi_{B \cap X_n} \cdot \d((f|B)^B).
$$
By a similar argument, Property (3) is a consequence of Property (2), the locality property, and the Leibniz rule for $\d$.
Now suppose that $X$ is separable.  Letting $\{x_n\}_{n=1}^\infty$ be a countable dense subset of $X$, put $X_0 = \emptyset$ and for each $n \in \N$, put
$$
X_n \;:=\;
B(x_n,1/2) \setminus \Big( \bigcup_{k=0}^{n-1} X_k \Big)
\; \textrm{ and } \; 
f_n \;:=\;
f - \inf_{X_n}f.
$$
Since each set $X_n$ has diameter at most $1$, we obtain
$$
\|f_n\|_\infty \;=\; 
\left|\sup_{X_n}f - \inf_{X_n}f\right| \;\leq\; 
L(f) \cdot \diam(X_n) \;=\; L(f).
$$
Invoking Lemma \ref{lemma_ctyderiv} and the estimate above, we now compute
\begin{eqnarray*}
\|\d_{X_n}(f|X_n) \|_{\mu,\,\infty} &=&
\|\d((f|X_n)^{X_n})\|_{\mu,\,\infty} \;=\;
\|\d((f_n|X_n)^{X_n})\|_{\mu,\,\infty} \\ &\leq&
%\|\d\|_{\rm op} \cdot \|(f_n|X_n)^{X_n}\|_{\Lip} \;=\;
\|\d\|_{\rm op} \cdot \|f_n|X_n\|_{\Lip} \;\leq\;
\|\d\|_{\rm op} \cdot L(f).
\end{eqnarray*}
This gives Property (4) and proves the theorem.
\end{proof}

\subsection{Pushforward Derivations}

Recall that for smooth manifolds $M$ and $N$ with respective tangent bundles $TM$ and $TN$, every smooth bijective map from $M$ to $N$ induces a pushforward operator on vector fields.  Indeed, for each smooth vector field $v : M \to TM$ we obtain a  new vector field $f_\#v : N \to TN$ from the rule 
$$
f_\#v(x) \;:=\; Df(f^{-1}(x)) \cdot v.
$$
A similar procedure holds for derivations, by means of pushforward measures.  Recall that on a measure space $(X,\mu)$, a set $Y$, and a map $T : X \to Y$, one defines the {\em pushforward measure} $T_\#\mu$ on $Y$ by the rule
$$
T_\#\mu(A) \;:=\; \mu(T^{-1}(A)).
$$
It is well known that if $\mu$ is a Borel measure and $T$ a Borel map, then $T_\#\mu$ is a Borel measure and  we have the ``change of variables'' formula
\begin{equation} \label{eq_changevari}
\int_Y \varphi \,d(T_\#\mu) \;=\;
\int_X \varphi \circ T \,d\mu
\end{equation}
whenever $\varphi : Y \to \R$ is a Borel function; see \cite[Thm 1.19]{Mattila}.

\begin{lemma} \label{lemma_pushfwd}
Let $X$, $Y$ be metric spaces, let $\mu$ be a Radon measure on $X$, and let $\pi \in \Lip(X;Y)$.  For each $\d \in \U(X,\mu)$, there is a unique $\pi_\#\d \in \U(Y,\pi_\#\mu)$ so that
\begin{equation} \label{eq_pushfwd}
\int_Y h \cdot (\pi_\#\d)f \,d(\pi_\#\mu) \;=\;
\int_X (h \circ \pi) \cdot \d(f \circ \pi) \,d\mu
\end{equation}
holds for all $h \in L^1(Y,\pi_\#\mu)$ and all $f \in \Lip_b(Y)$.  If $X$ is
separable, then %we have the bound
\begin{equation} \label{eq_pushfwdbd}
\|\pi_\#\d\|_{\rm op} \;\leq\; (1 \vee L(\pi)) \cdot \|\d\|_{\rm op}.
\end{equation}
\end{lemma}

We refer to $\pi_\#\d$ as the \emph{pushforward (derivation)} of $\d$ with respect to $\pi$.

\begin{proof}
%As a shorthand, p
Put $\nu := \pi_\#\mu$.  For $h \in L^1(Y,\nu)$, formula \eqref{eq_changevari} gives $h \circ \pi \in L^1(X,\mu)$, with $\|h\|_{\nu,\,1} = \|h \circ \pi\|_{\mu,\,1}$.  For each $f \in \Lip_b(Y)$, define a map $\lambda_{f,\,\pi} : L^1(Y,\nu) \to \R$ by
$$
\lambda_{f,\,\pi}(h) \;:=\;
\int_X (h \circ \pi) \cdot \d(f \circ \pi) \,d\mu.
$$
Clearly $\lambda_{f,\,\pi}$ is linear and bounded, so there is a unique function $(\pi_\#\d)f \in \Li(Y,\nu)$ which satisfies, for all $h \in L^1(Y,\nu)$, the identity
$$
\int_Y h \cdot (\pi_\#\d)f \,d\nu \;=\;
\int_X (h \circ \pi) \cdot \d(f \circ \pi) \,d\mu.
$$
As constructed, the map $\pi_\#\d : f \mapsto (\pi_\#\d)f$ satisfies formula \eqref{eq_pushfwd}.
Moreover, it is linear because $\d$ is linear; the same is true of the Leibniz rule.

To show that $\pi_\#\d$ is continuous, suppose $\{f_\a\}_{\a \in I}$ is a net in $\Lip_b(Y)$ that converges pointwise to $0$ %$f \in \Lip_b(Y)$ 
and so that $C := \sup_\a \|f_\a\|_{\Lip} < \infty$.  Clearly $f_\a \circ \pi$ converges pointwise to $0$, and from the estimates
\begin{equation}
\label{eq_temppushfwd}
\left.\begin{split}
\|f_\a \circ \pi\|_\infty &\;\leq\;
\|f_\a\|_\infty \;\leq\; C \\
\hspace{1.2in}
L(f_\a \circ \pi) &\;\leq\;
L(f_\a) \cdot L(\pi) \;\leq\; C \cdot L(\pi)
\hspace{.8in}
\end{split}\right\}
\end{equation}
the net $\{f_\a \circ \pi\}_{\a \in I}$ is bounded in $\Lip_b(X)$.  By Lemma \ref{lemma_weakstarlip} and the continuity of $\d$, we obtain $\d(f_\a \circ \pi) \wsto 0$ %\d(f \circ \pi) 
in $\Li(X,\mu)$.  Since $h \in L^1(Y,\nu)$ implies $h \circ \pi \in L^1(X,\mu)$, it follows that $(\pi_\#\d)f_\a \wsto 0$ %(\pi_\#\d)f 
in $\Li(Y,\nu)$.  

Let $f \in \Lip_b(Y)$.  If $X$ is separable, then by the estimates \eqref{eq_temppushfwd}, we obtain
\begin{eqnarray*}
\|(\pi_\#\d)f\|_{\mu,\,\infty} \;=\;
\|\lambda_{f,\,\pi}\|_{\rm op} &\leq&
\|h\|_{\nu,\,1} \cdot \|\d(f \circ \pi)\|_{\mu,\,\infty} \\ &\leq&
\|h\|_{\nu,\,1} \cdot \|\d\|_{\rm op} \cdot \|f \circ \pi\|_{\Lip} \\ &\leq&
\|h\|_{\nu,\,1} \cdot \|\d\|_{\rm op} \cdot C(1 \vee L(\pi)) \cdot \|f\|_{\Lip}.
\end{eqnarray*}
This implies inequality \eqref{eq_pushfwdbd}.  Lastly, suppose that
$\d' \in \U(Y,\nu)$ also satisfies formula \eqref{eq_pushfwd}.  By
linearity, we have, for all $h \in L^1(Y,\nu)$ and all $f \in \Lip_b(Y)$,
$$
\int_Y h \cdot (\pi_\#\d - \d')f \, \d\nu \;=\; 0.
$$
This means that $\pi_\#\d = \d'$, which gives the desired uniqueness.
\end{proof}

For $\pi \in \Lip(X;Y)$, note that $\U(X,\mu)$ is an $\Li(Y,\pi_\#\mu)$-module.  Indeed, for $\lambda \in \Li(Y,\pi_\#\mu)$, $f \in \Lip_b(X)$, and $\d \in \U(X,\mu)$, the action is given by
\begin{equation} \label{eq_pushfwdaction}
(\lambda \cdot \d) f \;:=\; (\lambda \circ \pi) \cdot \d f.
\end{equation}
Recall that an embedding $\pi : X \to Y$ is \emph{bi-Lipschitz} if $\pi$ and $\pi^{-1}$ are both Lipschitz maps; it is $\lambda$-bi-Lipschitz if $L(\pi) \leq \lambda$ and $L(\pi^{-1}) \leq \lambda$.

So if $\pi$ is bi-Lipschitz, then the proof of Lemma \ref{lemma_pushfwd} (with $\pi^{-1}$ for $\pi$) also shows that $\U(Y,\pi_\#\mu)$ is an $\Li(X,\mu)$-module.  Under an appropriate choice of measures, we now obtain a ``functorial'' property of pushforward derivations with respect to bi-Lipschitz embeddings.

\begin{cor}
Let $(X,\rho_X,\mu)$ and $(Y,\rho_Y,\nu)$ be metric measure spaces, with $\mu$ a Borel measure, and let $\pi : X \hookrightarrow Y$ be a bi-Lipschitz embedding.  If $\nu$ and $\pi_\#\mu$ are mutually absolutely continuous, then $\U(X,\mu)$ and $\U(Y,\nu)$ are isomorphic as $\Li(X,\mu)$-modules.
\end{cor}

\subsection{The Chain Rule}

On Euclidean spaces, derivations exhibit behavior similar to that of the differential operators $\{\partial_i\}_{i=1}^n$.  For instance, they satisfy a weak form of the Chain Rule from differential calculus.  To formulate this fact, recall that by Theorem \ref{thm_extlip}, each $\d x_j$ is a well-defined function in $\Li(\R^n,\mu)$ for $j = 1, 2, \ldots, n$.

\begin{lemma} \label{lemma_chainrule}
Let $\mu$ be a Radon measure on $\R^n$.  For each $f \in \Lip(\R^n)$, there exist %$g_f^i \in L^\infty(\R^n,\mu)$, $1 \leq i \leq n$, so that
functions $\{g_f^i\}_{=1}^n \subset \Li(\R^n,\mu)$ so that
\begin{eqnarray}
\label{eq_chainrulebd}
\|g_f^i\|_{\mu,\,\infty} &\leq& L(f) \\
\label{eq_chainrule}
\d f &=& \sum_{i=1}^n g_f^i \cdot \d x_i.
\end{eqnarray}
If $f$ is smooth, then $g_f^i = \partial_if$ for $i = 1, 2, \ldots, n$.
\end{lemma}

We refer to Lemma \ref{lemma_chainrule} as the Chain Rule for derivations.  Its proof uses a classical fact about approximation of smooth functions \cite[Thm II.4.3]{CourantHilbert}.

\begin{lemma} \label{lemma_polyapprox}
Let $\varphi \in C^\infty(\R^n)$.  For each compact subset $K$ of $\R^n$, there is a sequence of polynomials $\{P_m\}_{m =1}^\infty$ so that on $K$, we have the uniform convergence $P_m \to f$ and $\partial_i P_m \to \partial_i f$, for $1 \leq i \leq n$.
\end{lemma}

\begin{proof}[Proof of Lemma \ref{lemma_chainrule}]
Since $\R^n$ is a countable union of closed cubes $\{Q_k\}_{k=1}^\infty$, it suffices to show formula \eqref{eq_chainrule} for $\chi_{Q_k}\d$ in place of $\d$.  By the locality property, we therefore assume that $\d \in \U(Q_k,\mu)$.

We now argue by cases.  Formula \eqref{eq_chainrule} clearly holds when $f = x_j$, for each $j = 1, 2, \ldots, n$, and where $g_f^i$ is the Kronecker symbol $\e_{ij}$.  If $f$ is a polynomial, then for each $a \in \N$, the Leibniz rule implies a ``power rule''
$$
\d(x_j^a) \;=\; a x_j^{a-1} \cdot \d x_j
$$
which further implies formula \eqref{eq_chainrule}, with $g_f^i := \partial_if$.

We next assume that $f$ is a smooth Lipschitz function on $\R^n$.  By Lemma \ref{lemma_polyapprox}, there is a sequence of polynomials $\{P_m\}_{m=1}^\infty$ which converges uniformly to $f$ on $K$ and where $\{\nabla P_m\}_{m=1}^\infty$ converges uniformly to $\nabla f$.  This implies that $P_m \wsto f$ in $\Lip_b(Q_k)$, and by continuity of $\d$, we obtain $\d P_m \wsto \d f$ in $\Li(Q_k,\mu)$.  On the other hand, the convergence $\nabla P_m \to \nabla f$ is uniform, hence weak-$*$.  It follows that $\partial_iP_m \d x_i \wsto \partial_i f \d x_i$ in $\Li(Q_k,\mu)$ and by uniqueness of limits, we obtain formula \eqref{eq_chainrule}, where again $g_f^i := \partial_if$.

For the general case, let $\e > 0$ be arbitrary, let $\eta_\e$ be a smooth, symmetric mollifier, and consider convolutions $f_\e := f * \eta_\e$.  It is a fact \cite[Thm 4.2.1.1]{EvansGariepy} that if $f$ is continuous, then $f_\e$ converges locally uniformly to $f$.  Moreover, the bound $L(f_\e) \leq L(f)$ follows from the computation
\begin{eqnarray*}
|f_\e(x) - f_\e(y)| &\leq&
\int_{\R^n} \eta_\e(z) \cdot |f(x-z) - f(y-z)| \,dz \\ &\leq&
\int_{\R^n} \eta_\e(z) \cdot L(f) \cdot |(x-z) - (y-z)| \,dz \,\leq\,
L(f) \cdot |x-y|.
\end{eqnarray*}
This implies that $f_\e \wsto f$ in $\Lip_b(Q_k)$ and from the continuity of $\d$, we also have $\d f_\e \wsto \d f$ in $\Li(Q_k,\mu)$.

However, note that formula \eqref{eq_chainrule} holds for each $f_\e$, where $g_f^i = \partial_if_\e$, and note that $\{\partial_if_{1/a}\}_{a=1}^\infty$ is a bounded sequence in $\Li(\R^n,\mu)$, for each $i$.  It follows from the Banach-Alaoglu Theorem that there are weak-$*$ convergent subsequences $\{\partial_if_{1/{a_b}}\}_{b=1}^\infty$ with weak-$*$ limits $g_f^i$.

By uniqueness of limits, formula \eqref{eq_chainrule} holds for $f$ with these choices of $g_f^i$.  Since the norm on $\Li(\R^n,\mu)$ %a dual Banach space 
is lower semi-continuous (with respect to the weak-$*$ topology), formula \eqref{eq_chainrulebd} follows from
$$
\|g_f^i\|_{\mu,\,\infty} \;\leq\; 
\liminf_{b \to \infty} \|\partial_if_{1/a_b}\|_{\mu,\,\infty} \;\leq\; L(f)
\qedhere.
$$
\end{proof}

The next corollary is a criterion for detecting nonzero derivations on $\R^n$.  It follows directly from Lemma \ref{lemma_chainrule}, so we omit the proof.

\begin{cor} \label{cor_zerocoord}
Let $\mu$ be a Radon measure on $\R^n$ and let $\d \in \U(\R^n,\mu)$.  If $\d x_j = 0$ holds for each $j = 1, 2, \ldots, n$, then $\d = 0$.
\end{cor}

%\vfill
%\pagebreak
%=========================================================================

\section{Derivations on $1$-Dimensional Sets} \label{sect_1dim}

Adapting the terminology in \cite{FalconerBook}, a subset of $\R^n$ is called a \emph{$k$-set} if it is $\H^k$-measurable and has $\sigma$-finite $\H^k$-measure.  In this section we will focus on the following fact about measures concentrated on $1$-sets and their derivations.

\begin{thm} \label{thm_1sets}
Let $\mu$ be a Radon measure on $\R^n$.  If $\mu$ is concentrated on a $1$-set, then the module $\U(\R^n,\mu)$ has rank at most $1$.
\end{thm}

The proof uses facts from geometric measure theory, which are discussed in \S\ref{sect_derivsGMT}.  We begin with a special case.

\subsection{The Case of $\R$}

Using the Borel regularity of Lebesgue measure, we prove Theorem \ref{thm_rigiditylowdim} for $k=1$, which is a characterization of $\U(\R,\mu)$.  To this end, recall that every Radon measure $\mu$ on $\R$ admits a decomposition $\mu = \mu_{AC} + \mu_S$, where $\mu_{AC}$ is absolutely continuous to $m_1$ and $\mu_S$ is singular to $m_1$ \cite[Thm 3.8]{Folland}.

\begin{lemma} \label{lemma_dim1}
Let $\mu$ be a Radon measure on $\R$.  If $\mu_S$ is concentrated on a Lebesgue null set $E$,
then for all $\d \in \U(\R,\mu)$ and all $f \in \Lip_b(\R)$,
\begin{equation} \label{eq_dim1}
\d f(x) \;=\; 
\begin{cases}
\d(\id_R)(x) \cdot f'(x), & \textrm{for } x \in \R \setminus E \\
0, & \textrm{for } x \in E
\end{cases}.
\end{equation}
where $f'$ is the classical derivative of $f$.  Moreover, as $\Li(\R,\mu)$-modules, 
$$
\U(\R,\mu) \;\cong\;\Li(\R,\mu_{AC}).
$$
\end{lemma}

\begin{proof}
Let $\d \in \U(\R,\mu)$ and $f \in \Lip_b(\R)$ be arbitrary.  For subsets of $\R \setminus E$ we have $\mu = \mu_{AC}$, %As $\mu_{AC}$ is absolutely continuous to $m_1$, it follows from 
so by Rademacher's theorem, $f$ is differentiable $\mu$-a.e.\ on $\R \setminus E$.  The Chain Rule for derivations then implies that
$$
\d f(x) \;=\; \d(\id_\R)(x) \cdot f'(x)
$$
for $\mu_{AC}$-a.e.\ $x \in \R$ and hence for $\mu$-a.e.\ $x \in \R \setminus E$.

To show $\chi_E \cdot \d f = 0$, assume by locality (Theorem \ref{thm_locality}) that $E$ is bounded. 
In particular, let $E \subset [0,1]$, so $\chi_E \d \in \U([0,1],\mu)$.

Since $m_1(E) = 0$, %$E$ is a Lebesgue null set, 
for each $j \in \N$ there is a %bounded 
open set $O_j$ %in $\R$ 
so that $E \subset O_j$ and $m_1(O_j) < 2^{-j}$.  We next define %a sequence of Lipschitz 
functions $\varphi_j : [0,1] \to \R$ by the formula
$$
\varphi_j(x) \;:=\; \int_0^x (1 -\chi_{O_j}) \,dm_1.
$$
Clearly $\|\varphi_j\|_{\Lip} \leq 1$ holds, for each $j$.  Estimating further, we see that
$$
0 \;\leq\;
x - \varphi_j(x) \;=\;
\int_0^x \chi_{O_j} \,dm_1 \;\leq\;
m_1(O_j) \;\leq\; 2^{-j},
$$
and hence %from which we conclude that 
$\{\varphi_j\}_{j=1}^\infty$ converges pointwise to the identity on $\R$.  By Lemma \ref{lemma_weakstarlip}, this is equivalent to weak-$*$ convergence in $\Lip_b([0,1])$, and by %the
 continuity, %property of derivations, 
we obtain %the convergence 
$\d\varphi_j \wsto \d(\id_\R)$ in $\Li([0,1],\mu)$.

However, if $O_j'$ is a connected component of $O_j$, then by construction, $\varphi_j|O_j'$ is constant for each $j$.  The locality property implies that $\d\varphi_j(x) = 0$ holds for $\mu$-a.e.\ $x \in O_j' \cap [0,1]$, for each $j$, and hence %therefore that 
$\chi_E \cdot \d\varphi_j = 0$.  By continuity we obtain $\chi_E \cdot \d(\id_\R)= 0$, and by the Chain Rule we further obtain $\chi_E \cdot \d f = 0$.  This proves formula \eqref{eq_dim1}.

Consider maps $S : \U(\R,\mu) \to  \Li(\R,\mu_{AC})$ and $T : \Li(\R,\mu_{AC}) \to \U(\R,\mu)$ given by $S(\d) := \d(\id_\R)$ and $T(\lambda) := \lambda \cdot (d/dx)$.
Clearly, $S$ and $T$ are homomorphisms of $\Li(\R,\mu)$-modules.  Using formula \eqref{eq_dim1} and the previous estimates, 
$$
(T \circ S)(\d) \;=\;
T\big(\chi_{\R \setminus E} \cdot \d(\id_\R)\big) \;=\;
\chi_{\R \setminus E} \cdot \d(\id_\R) \cdot (d/dx) \;=\; \d
$$
and hence $S \circ T = \id_{\U(\R,\mu)}$.  A similar computation gives $S \circ T = \id_{\Li(\R,\mu)}$.
\end{proof}

\subsection{The General Case} \label{sect_derivsGMT}

We now introduce two types of sets in $\R^n$.

\begin{defn}
Let $k \in \N$
A subset $E$ in $\R^n$ is \emph{$k$-rectifiable} if, for some $\lambda \in (1,\infty)$ it admits a
$\H^k$-measurable decomposition of the form
\begin{equation} \label{eq_defnrect}
E \;=\; N \cup \bigcup_{i=1}^\infty f_i(A_i),
\end{equation}
where $\H^k(N) = 0$ %$N$ is a $\H^k$-null set in $\R^n$ 
and where, for each $i \in \N$, $A_i$ is a compact subset of $\R^k$ with $m_k(A_i) > 0$ and $f_i : A_i \to \R^n$ is a $\lambda$-bi-Lipschitz embedding.

A subset $F$ in $\R^n$ is \emph{purely $k$-unrectifiable} if $\H^k(E \cap F) = 0$ holds for all $k$-rectifiable sets $E$ in $\R^n$.
\end{defn}

\begin{rmk}
This definition differs substantially from the standard definition of $k$-rectifiability; see \cite[Defn 15.3]{Mattila} or \cite[Defn 3.2.14(1)]{Federer}.  However, by \cite[Lem 3.2.18]{Federer} these definitions are equivalent.
\end{rmk}

%Below, we see that 
Indeed, each $k$-set is a union of sets of the above types \cite[Thm 15.6]{Mattila}.

\begin{lemma} \label{lemma_rectdecomp}
Let $n \in \N$ and let $k$ be an integer in $[0,n]$.  If $A$ is a $k$-set in $\R^n$, then there is a $\H^k$-measurable decomposition $A = E \cup F$, where $E$ is $k$-rectifiable and $F$ is purely $k$-unrectifiable.
\end{lemma}

To prove Theorem \ref{thm_1sets}, we will use an alternative characterization of purely $k$-unrectifiable subsets in $\R^n$ \cite[Thm 18.1]{Mattila}.  Below, $\G(n;k)$ denotes the space of $k$-dimensional subspaces of $\R^n$ and ``almost everywhere'' refers to the Haar measure on $\G(n;k)$.  When $k = 1$, this measure %the Haar measure on $\G(n;1)$ 
is equivalent to (normalized) surface measure on the half-sphere $\{ x \in \mathbb{S}^{n-1}: x_n > 0\}$.

\begin{thm}[Besicovitch-Federer] \label{thm_projections}
For $0 \leq k \leq n$, let $F$ be a $k$-set in $\R^n$.  Then $F$ is purely $k$-unrectifiable if and only if 
for a.e.\ $V \in \G(n;k)$, the image $\proj_V(F)$ has $\H^k$-measure zero.
\end{thm}

In the remainder of the section, we assume $k = 1$.  The proof of Theorem \ref{thm_1sets} is split into two cases.

\begin{lemma} \label{lemma_pure1unrect}
Let $\mu$ be a Radon measure on $\R^n$.  If $\mu$ is concentrated on a 
purely $1$-unrectifiable set of Hausdorff dimension (at most) one, then 
$\U(\R^n,\mu) = 0$.
\end{lemma}

\begin{proof}%[Proof of Lemma \ref{lemma_pure1unrect}]
By Theorem \ref{thm_projections}, $F$ satisfies $\H^1(\proj_V(F)) = 0$ for a.e.\ $V \in \G(n;1)$.  In particular, there exist subspaces $\{V_i\}_{i=1}^n \subset \G(n;1)$ whose union spans $\R^n$ and so that $\H^1(\proj_{V_i}(F)) = 0$ holds, for each $1 \leq i \leq n$.  

Put $p^i := \proj_{V_i}$. Since $\H^1(p^i(F)) = 0$, we observe that $p^i_\#\mu$ is singular to $\H^1 \lfloor V_i$. By identifying $V_i$ with $\R$, we further observe that $\U(V_i, p^i_\#\mu) = 0$.

We claim that $\d p^i = 0$ holds for all $\d \in \U(\R^n,\mu)$; if not, the set 
$$
F_i \;:=\; \{ x \in F : \d p^i(x) \neq 0\}
$$
has positive $\mu$-measure.  For each bounded domain $\Omega$ in $\R^n$, formula \eqref{eq_pushfwd} implies
$$
0 \;<\;
\int_{\R^n} \chi_{\Omega \cap F_i} \cdot \d p^i \, d\mu \;=\;
\int_{V_i} \chi_{p^i(\Omega \cap F_i)} \cdot p^i_\#\d(\id_\R) \,d(p^i_\#\mu).
$$
However, the rightmost term is zero because $p^i_\#\d \in \U(V_i,p^i_\#\mu)$; therefore $p^i_\#\d$ is zero.  Since $\Omega$ was arbitrary, the claim follows.

Lastly, the linear functions $\{p^i\}_{i=1}^n$ generate the coordinate functions $\{x_i\}_{i=1}^n$.  This implies that $\d x_i = 0$ holds $\mu$-a.e.\ for each $i$, so by the Chain Rule for derivations, we conclude that $\d = 0$.
\end{proof}

In the case of $1$-rectifiable sets, the next lemma extends a result of Weaver \cite[Thm 38]{WeaverED} to arbitrary Radon measures on $\R^n$.

\begin{lemma} \label{lemma_1rect}
Suppose that $\mu$ is a Radon measure on $\R^n$ that is concentrated on a $1$-rectifiable set $E$.  If $\U(\R^n,\mu)$ is nontrivial, then $\U(\R^n,\mu)$ has rank-$1$.
\end{lemma}

One can further show that the generator of $\U(\R^n,\mu)$ is given by
$f \mapsto \chi_E \cdot D_{ap}f$,
where  $D_{ap}f$ is the \emph{approximate derivative} of the restriction 
$f|E$.  For more about approximate derivatives, see \cite[Sect 3.1.22]{Federer}.

\begin{proof}%[Proof of Lemma \ref{lemma_1rect}]
Let $E$ be a $1$-rectifiable set on which $\mu$ is concentrated.  As a first case, assume that $E = f(A)$, where $A \subset \R$ satisfies $m_1(A) > 0$ and where $f : A \to \R^n$ is a bi-Lipschitz embedding.  By Lemma \ref{lemma_pushfwd}, for each $\d \in \U(\R^n,\mu)$ there is a unique element $f^{-1}_\#\d$ in $\U(A, \nu)$, where $\nu := f^{-1}_\#\mu$.

If $\nu_{AC} = 0$, then $f^{-1}_\#\d = 0$ follows from Lemma \ref{lemma_dim1}.  Let $h \in L^1(\R^n,\mu)$ and $\varphi \in \Lip_b(\R^n)$ be arbitrary.  By formula \eqref{eq_pushfwd} and the previous identity,
\begin{eqnarray*}
\int_{\R^n} h \cdot \d \varphi \,d\mu &=&
\int_A (h \circ f^{-1}) \cdot f^{-1}_\#\d(\varphi \circ f^{-1}) \,d\nu \\ &=&
\int_A (h \circ f^{-1}) \cdot \lambda \cdot \Big(\chi_{A'} \frac{d}{dx}\Big)(\varphi \circ f^{-1}) \,d\nu \\ &=&
\int_{\R^n} h \cdot (\lambda \circ f) \cdot f_\#\Big(\chi_{A'} \frac{d}{dx}\Big)\varphi \,d\mu.
\end{eqnarray*}
so $f_\#(\chi_{A'} d / dx)$ generates $\U(\R^n,\mu)$.

For the general case, let $E = \bigcup_{i=1}^\infty f_i(A_i)$, where each $A_i$ is compact and each $f_i : A_i \to \R^n$ is $2$-biLipschitz.  Indeed, if $N$ is an $\H^1$-null set in $\R^n$, then $N$ is purely $1$-unrectifiable and by Lemma \ref{lemma_pure1unrect}, $\U(N,\mu) = 0$.

Put $E_i := f_i(A_i)$ and $\mu_i := \mu \lfloor E_i$.  By the previous case, 
$\d_i := (f_i)_\#(\chi_{A_i'}d/dx)$
generates $\U(E_i,\mu)$, where each $A_i'$ is a subset of $A_i$ on which $(f_i^{-1})_\#\mu_i$ is concentrated.  Moreover, from estimate \eqref{eq_pushfwdbd}, we have
$$
\|\d_i\|_{\rm op} \;\leq\; (1 \vee L(f_i)) \cdot \|\chi_{A_i'} d/dx\|_{\rm op} \;\leq\; 2.
$$
By Lemma \ref{lemma_gluing}, the map $\d_0 := \sum_{i=1}^\infty \chi_{E_i}\d_i$ is a well-defined element of $\U(\R^n,\mu)$.  For each $i \in \N$ and each $\d \in \U(\R^n,\mu)$, put
$$
\lambda_i := ((f_i^{-1})_\#\d)(\id_\R) \; \textrm{ and } \;
\lambda := \sum_{i=1}^\infty \chi_{E_i}\lambda_i.
$$
By an analogous argument as above, we obtain $\d = \lambda \cdot \d_0$.  In addition, for each $i \in \N$, the set $E_i$ is bounded and hence
\begin{eqnarray*}
\|\chi_{E_i} \lambda_i\|_{\mu,\,\infty} &\leq&
\|(f_i^{-1})_\#\d\|_{\rm op} \cdot \|\id_{E_i}\|_{\Lip} \\ &\leq&
\big( 1 \vee L(f_i^{-1})\big) \cdot \|\d\|_{\rm op} \cdot (1 \vee \diam(E_i)) \;\leq\;
2 \cdot \|\d\|_{\rm op}.
\end{eqnarray*}
By Part (1) of Lemma \ref{lemma_ctyderiv}, $\d$ is a bounded operator, so $\lambda \in \Li(\R^n,\mu)$.
\end{proof}

\begin{proof}[Proof of Theorem \ref{thm_1sets}]
%To begin, l
Let $A$ be a $1$-set on which $\mu$ is concentrated.  By Lemma \ref{lemma_rectdecomp}, we have the $\H^1$-decomposition $A = E \cup F$, where $E$ is $1$-rectifiable and $F$ is purely $1$-unrectifiable.

If $\mu(F) > 0$, then by the locality property and by Lemma \ref{lemma_pure1unrect}, the set $\{\chi_F\d_1, \chi_F\d_2\}$ is linearly dependent in $\U(\R^n,\mu)$.  It follows from Lemma \ref{lemma_lindepsubset} that $\{\d_1,\d_2\}$ is also linearly dependent in $\U(\R^n,\mu)$.

If instead $\mu(F) = 0$, then $\mu$ is concentrated on $E$ and hence $\mu = \mu \lfloor E$.
%This implies that $\chi_E \d_i = \d_i$ for $i = 1, 2$.
Let $\d_0$ be the generator of $\U(\R^n,\mu \lfloor E)$. 
%, as defined in the proof of Lemma \ref{lemma_1rect}.
For each $i = 1,2$ there is a nonzero function $\lambda_i \in \Li(\R^n,\mu)$ so that 
$\chi_E \d_i = \d_i = \lambda_i \d_0$.  
We now put
$$
\Lambda_1(x) \;:=\;
\chi_{\spt(\lambda_i)} \cdot \left[1 \wedge \frac{\lambda_2(x)}{\lambda_1(x)} \right] 
\; \textrm{ and } \; 
\Lambda_2(x) \;:=\; 
\chi_{\spt(\lambda_2)} \cdot \left[1 \wedge \frac{\lambda_1(x)}{\lambda_2(x)}\right].
$$
By construction, $\Lambda_1 \d_1 - \Lambda_2 \d_2 = 0$.  Neither $\Lambda_1$ nor $\Lambda_2$ is zero, otherwise one of $\lambda_1$ and $\lambda_2$ is zero, which is a contradiction.
\end{proof}

%\vfill
%\pagebreak
%=========================================================================
\section{Derivations on $2$-Dimensional Sets} \label{sect_2dim}

Let $\mu$ be a Radon measure on $\R$.  As a consequence of Theorem \ref{lemma_dim1}, if $\mu$ is singular to Lebesgue measure, then $\U(\R,\mu)$ has rank $0$.  A similar statement holds true for Radon measures on $\R^2$.%, and it implies the $k=2$ case of Theorem \ref{thm_rigiditylowdim}.

\begin{thm} \label{thm_2dimLI}
Let $\mu$ be a Radon measure on $\R^2$.  If $\mu$ is singular to Lebesgue measure, then the module $\U(\R^2,\mu)$ has rank $1$.
\end{thm}

%For Lebesgue singular measures $\mu$, r
Recall that the proof of Theorem \ref{lemma_dim1} consists of selecting open covers for a Lebesgue null set (on which $\mu$ is concentrated).  From these covers, one constructs a sequence of uniformly Lipschitz functions on $\R$ that converges to the identity.

The proof of Theorem \ref{thm_2dimLI} follows similar ideas.  However, in order to construct analogous functions, we will use recent results of Alberti, Cs\"ornyei, and Preiss about the structure of Lebesgue null sets \cite{ACP}.  This provides covers of such sets with a suitable geometry.  In what follows, we refer to Lebesgue null sets simply as null sets, Lebesgue singular measures as singular measures, and so on.

\subsection{%Preliminaries: 
Null Sets in $\R^2$} \label{subsect_ACP}

We begin with a few definitions from \cite{ACP}.

\begin{defn}
An \emph{$x_1$-curve} in $\R^2$ is a graph of the form
$$
\gamma^1(f) \;:=\; \{ (t, f(t)) \,:\, t \in \R \},
$$
where $f : \R \to \R$ is $1$-Lipschitz.  We call $f$ the \emph{(Lipschitz) parametrization} of $\gamma = \gamma^1(f)$.
For $\d > 0$, an \emph{$x_1$-stripe of thickness $\d$} is a set of the form
$$
\n^1(g;\d) \;:=\; \{ (t,y) \,:\, |y - g(t)| \leq \d / 2 \}
$$
where $g : \R \to \R$ is also $1$-Lipschitz.  An \emph{$x_2$-curve} and a \emph{$x_2$-stripe (of thickness $\d$)} are similarly defined.
\end{defn}
%To clarify, in \cite{ACP} such sets are called \emph{$x$-curves}, \emph{$x$-stripes}, \emph{$y$-curves}, and \emph{$y$-stripes}, respectively.

We now state a covering theorem for null sets in $\R^2$ \cite[Thm 2]{ACP}.  The case of compact null sets follows from the proof in \cite[pp.\ 4-5]{ACP}.

\begin{thm}[Alberti-Cs\"ornyei-Preiss, 2005] \label{thm_covering}
Let $E$ be a null set in $\R^2$.
Then there is a decomposition $E = E^1 \cup E^2$, where each set $E^i$ satisfies the following property: for each $\e > 0$, there are $x_i$-stripes $\{\n^i(f_j^i;\d_j^i)\}_{j=1}^\infty$ so that their union covers $E^i$ and so that $\sum_{j=1}^\infty \d_j^i < \e$.

If $E$ is compact, then for each $\e > 0$, there exist $N \in \N$ and $\d > 0$ so that each $E^i$ can be covered by $N$ many $x_i$-stripes $\n^i(f_j^i;\d)$, with $N \cdot \d < \e$, and so that each $f_j^i$ is piecewise-linear with finitely many points of non-differentiability.
\end{thm}

\begin{rmk}
Strictly speaking, the argument in \cite{ACP} only shows that shows that for each $\e > 0$, the null set $E$ can be covered by unions of $x_1$- and $x_2$-stripes $\{\n_i^{1,\e}\}_{i=1}^\infty$ and $\{\n_j^{2,\e}\}_{j=1}^\infty$, respectively, with the desired properties.  However, one easily obtains the subsets $E^1$ and $E^2$ by putting
$$
E^1 \;:=\; \bigcap_{k=1}^\infty \bigcup_{i=1}^\infty \n_i^{1,1/k} \; \textrm{ and } \;
E^2 \;:=\; E \setminus E^1.
$$
\end{rmk}

The next theorem will be a crucial step in the proof of Theorem \ref{thm_2dimLI}.

\begin{thm} \label{thm_disjointcovering}
Let $E$ be a compact null set in $\R^2$.  In addition to the properties given in Theorem \ref{thm_covering}, for $i = 1, 2$ and for each $\e > 0$ the covering $x_i$-stripes for $E^i$ can be chosen to have pairwise-disjoint interiors.
\end{thm}

To prove the theorem, we first require a lemma.  It guarantees that $x_i$-curves associated to the covering $x_i$-stripes can be chosen without transversal crossings.  (The basic idea to is to take pointwise maxima among the collection of $x_i$-curves, and iterate.)

\begin{lemma} \label{lemma_nocrossings}
Let $i = 1, 2$.  For each collection of $x_i$-curves $\{\a_j\}_{j=1}^N$, there is a collection of $x_i$-curves $\{\b_j\}_{j=1}^N$, with $\b_j := \gamma^i(f_j)$, so that
\begin{equation} \label{eq_sameunion}
\a_1 \cup \ldots \cup \a_N \;=\; \b_1 \cup \ldots \cup \b_N
\end{equation}
and so that, for all $t \in \R$ and all $1 < j \leq N$, we have
\begin{equation} \label{eq_curveorder}
f_{j-1}(t) \;\leq\; f_j(t).
\end{equation}
If the curves $\{\a_j\}_{j=1}^N$ are piecewise-linear, then so are the curves $\{\b_j\}_{j=1}^N$.
\end{lemma}

\begin{proof}[Proof of Lemma \ref{lemma_nocrossings}]
By symmetry, we assume that $i = 1$.   We argue by induction, and for $N = 1$, the lemma trivially holds with $\b_1 = \a_1$.  

Fix $n \in \N$ and let $\{\a_j\}_{j=1}^{n+1}$ be any collection of $x_1$-curves.  By the induction hypothesis, for $\{\a_j\}_{j=1}^n$ there are $x_1$-curves $\{b_j\}_{j=1}^n$ which satisfy \eqref{eq_sameunion} and \eqref{eq_curveorder}.
For $j = 1, 2, \ldots, n$, let $g_j : \R \to \R$ be the parametrization of $b_j$, so $g_j \leq g_{j+1}$, and let $g_{n+1}$ be the parametrization of $\a_{n+1}$.  We now define
$$
h_j \;:=\;
\begin{cases}
g_1 \vee g_{n+1}, & j=1 \\
g_j \vee h_{j-1}, & 1 < j \leq n
\end{cases}, \; \;
f_j \;:=\;
\begin{cases}
g_1 \wedge g_{n+1}, & j=1 \\
g_j \wedge h_{j-1}, & 1 < j \leq n \\
h_n, & j = n+1.
\end{cases}
$$

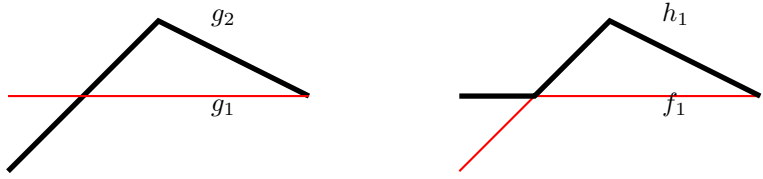
\begin{figure}[htp]
\centering
\begin{pspicture}(0,0.5)(12,2.5)
\psline[linewidth=2pt](1,0.5)(3,2.5)(5,1.5)
\put(3.7,2.5){$g_2$}
%\uput[45](3.7,2.5){$g_2$}
\psline[linecolor=red](1,1.5)(5,1.5)
\put(3.7,1.3){$g_1$}
%\uput[45](3.7,1.3){$g_1$}
\psline[linecolor=red](7,0.5)(8,1.5)(11,1.5)
\put(9.7,1.3){$f_1$}
%\uput[45](9.7,1.3){$f_1$}
\psline[linewidth=2pt](7,1.5)(8,1.5)(9,2.5)(11,1.5)
\put(9.7,2.5){$h_1$}
%\uput[45](9.7,2.5){$h_1$}
\end{pspicture}
\caption{Uncrossing $x_1$-curves, for $n=2$.}
\end{figure}
By construction, for each $x \in \R$ and each $k$, there is a unique index $j$ so that $g_k(x) = f_j(x)$.  Putting $\b_j := \gamma^1(f_j)$, we see that equation \eqref{eq_sameunion} holds for the collections of curves
$\{\a_j\}_{j=1}^{n+1}$ and $\{\b_j\}_{j=1}^{n+1}$.

For each $j$, we have $g_j \leq g_{j+1}$ by hypothesis and $h_j \leq h_{j+1}$ by construction, so
\begin{eqnarray*}
f_j &=&
g_j \wedge h_{j-1} \;\leq\;
g_j \;\leq\;
g_{j+1}, \\
f_j &=&
g_j \wedge h_{j-1} \;\leq\;
g_j \vee h_{j-1} \;=\;
h_j.
\end{eqnarray*}
By definition of $f_{j+1}$, it follows that inequality \eqref{eq_curveorder} holds for all $j$.
\end{proof}

For Theorem \ref{thm_disjointcovering}, the basic idea is that if $x_i$-stripes overlap, then by uncrossing the corresponding $x_i$-curves, the top stripe can then be ``pushed'' off the bottom one.

\begin{proof}[Proof of Theorem \ref{thm_disjointcovering}]
Let $\e > 0$ be given.  By Theorem \ref{thm_covering}, for each set $E^i$ there is a $\d > 0$ and there are $x_i$-stripes $\{\n^i(g_j^i;\d)\}_{j=1}^N$ so that their union covers $E^i$, so that each $g_j^i$ is piecewise-linear, and so that $N \cdot \d < \e$.  The argument is symmetric, so we assume that $i = 1$.  %As a shorthand, w
We also write $g_j := g_j^i$.

By Lemma \ref{lemma_nocrossings}, there are $1$-Lipschitz %, piecewise-linear 
functions  $\{f_j\}_{j=1}^N$ so that the $x_1$-curves
 $\{\gamma^1(g_j)\}_{j=1}^N$ and $\{\gamma^1(f_j)\}_{j=1}^N$
satisfy equations \eqref{eq_sameunion} and \eqref{eq_curveorder}.  Put
$$
h_{1,j} \;:=\;
\left\{\begin{array}{ll}
f_1, & j = 1 \\
f_j \vee (f_1 + \d), & j > 1.
\end{array}\right.
$$

\begin{figure}[htp]
\centering
\begin{pspicture}(0,0.5)(12,3.5)
%BEFORE BUMP
\psline[]{<->}(0.75,1.5)(0.75,2.5)
%\uput[45](0.5,2){$\d$}
\put(0.4,2){$\d$}
\psline[](1,1.5)(3,2.5)(5,1.5)
\psline[](1,0.5)(3,1.5)(5,0.5)
\psline[linestyle=dashed](1,1)(3,2)(5,1)
%\uput[45](5.8,1){$h_{1,1}=f_1$}
\put(5.2,1){$h_{1,1}=f_1$}
\psline[linecolor=red](1,2.5)(3,2.5)(5,3.5)
\psline[linecolor=red](1,1.5)(3,1.5)(5,2.5)
\psline[linestyle=dashed,linecolor=red](1,2)(3,2)(5,3)
\put(5.2,3){$f_2$}
%\uput[45](5.2,3){$f_2$}
%AFTER BUMP
\psline[](7,1.5)(9,2.5)(11,1.5)
\psline[](7,0.5)(9,1.5)(11,0.5)
\psline[linestyle=dashed](7,1)(9,2)(11,1)
\put(11.2,1){$h_{1,1}$}
%\uput[45](11.4,1){$h_{1,1}$}
\psline[linecolor=red](7,2.5)(9,3.5)(10,3)(11,3.5)
\psline[linecolor=red](7,1.55)(9,2.55)(10,2.05)(11,2.55)
\psline[linestyle=dashed,linecolor=red](7,2)(9,3)(10,2.5)(11,3)
\put(11.2,3){$h_{1,2}$}
%\uput[45](11.4,3){$h_{1,2}$}
\end{pspicture}
\caption{Choosing stripes with pairwise-disjoint interiors.}
\end{figure}
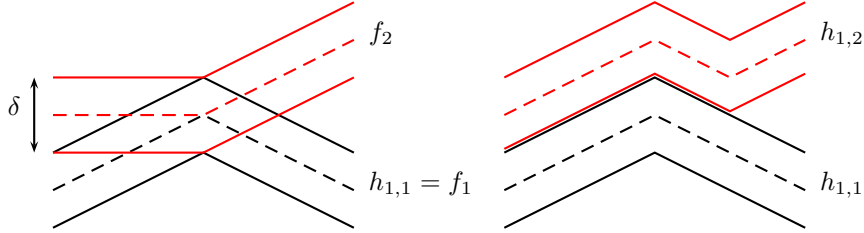

By construction, for $j > 1$ none of the stripes $\n^1(h_{1,j};\d)$ meets the interior of the stripe $\n^1(h_{1,1};\d)$.  It also remains that $h_{1,j} \leq h_{1,j+1}$.  We now claim that
\begin{equation} \label{eq_stillcover}
\bigcup_{j=1}^N \n^1(f_j;\d) \;\subset\;
%S_1 \;:=\;
\bigcup_{j=1}^N \n^1(h_{1,j};\d).
\end{equation}
Fix $(t,y) \in \n^1(f_j;\d)$.  In the case $h_{1,j}(t) = f_j(t)$, it follows by construction that $(t,y) \in \bigcup_{j=1}^N \n^1(h_{1,j};\d)$.  If instead $h_{1,j}(t) = f_1(t) + \d$, the point $(t,y)$ satisfies
\begin{eqnarray}
\d/2 &<& |y - f_1(t)| \\
\label{eq_bumpedstripe}
f_j(t) &<& f_1(t) + \d \;=\; h_{1,j}(t)
\end{eqnarray}
where again, $j > 1$.  From inequality \eqref{eq_bumpedstripe} we obtain
$$
y - (f_1(t) + \d) \;\leq\;
y - f_j(t) \;\leq\;
\d/2.
$$
Since $j > 1$ and $(t,y) \in \n^1(f_j;\d)$, we may further assume by inequality \eqref{eq_curveorder} that $y - f_1(t) > \d/2$.  This in turn gives the estimate
$$
-\d/2 \;=\;
\d/2 - \d \;<\;
y - f_1(t) - \d
$$
from which we obtain $(t,y) \in \n^1(h_{1,j};\d)$.  This gives the set inclusion \eqref{eq_stillcover}.

We now iterate the argument.  For $k = 1, 2, \ldots, N$, %each integer $k = 1, 2, \ldots, N$,
put
$$
h_{k,j} \;:=\;
\left\{\begin{array}{ll}
h_{k-1,\,j}, & j \leq k \\
f_k \vee (h_{k,\,k} + \d), & j > k.
\end{array}\right.
$$
Arguing similarly, we see that inclusion \eqref{eq_stillcover} holds with $h_{k,\,j}$ in place of $h_{1,\,j}$ %, that the stripes $\{\n^1(h_{k,\,j};\d)\}_{j=1}^k$ have pairwise-disjoint interiors 
and that, for $k \leq j \leq N$, none of the stripes $\n^1(h_{k,\,j};\d)$ meets the interiors of the previous $k$ many $x_1$-stripes.  Thus $\{\n^1(h_{j,\,j};\d)\}_{j=1}^N$ is the desired collection of $x_1$-stripes for $E^1$.
\end{proof}

Before returning to derivations, we recall a fact \cite[Rmk 3(ii)]{ACP} about the geometry of $E^1$ and $E^2$.  For completeness, we prove it below.

\begin{lemma} \label{lemma_curveintersect}
Let $E$ be a null set in $\R^2$ and let $L \in (0,1)$.  For $\{i,k\} = \{1,2\}$, if $E^i$ is the subset from Theorem \ref{thm_covering} and if $g : \R \to \R$ is $L$-Lipschitz, then
$$
\H^1(E^i \cap \gamma^k(g)) \;=\; 0.
$$
\end{lemma}

\begin{proof}
By Theorem \ref{thm_covering}, for each $\e > 0$ there are $x_i$-stripes $\n^i_j := \n^i(f_j^i;\d_j^i)$, $j \in \N$
so that $E^i \subset \bigcup_j \n^i_j$ and so that $\sum_i \d_j^i < \e$.   Clearly, the same union of $x_i$-stripes also covers the subset $E^i \cap \gamma^k(g)$.

For each $j \in \N$, let $p_j$ be the point in $\gamma^k(g) \cap \n^i_j$ with least $x_k$-coordinate.  Note that $\gamma^k(g) \cap \n^i_j$ can be covered by the set %of the form 
$C(p_j) \cap \n^i_j$, where $C(p_j)$ is a one-sided cone with vertex $p_j$, direction $\vec{e}_k$, and opening angle $2\arctan(1/L)$.  In particular, %the set 
$C(p_j) \cap \n^i_j$ has diameter at most $C \cdot \d_j^i$, where $C$ is a positive constant depending only on $L$.

In this way we cover $E^i \cap \gamma^k(g)$ with open sets $\{O_j\}_{j=1}^\infty$, each of %which has 
diameter at most $2C \cdot \d_j^i$ and hence at most $2C \cdot \e$.  We now estimate:
$$
\H^1(E^i \cap \gamma^k(g)) \;\leq\;
\limsup_{\e \to 0} \sum_{j=1}^\infty \diam(O_j) \;\leq\;
\limsup_{\e \to 0} \sum_{j=1}^\infty 2C \cdot \d_j^i \;<\;
2C \cdot \e.
$$
Since $\e > 0$ was arbitrary, the lemma follows.
\end{proof}

\subsection{Approximating the Coordinate Functions} 
\label{subsect_pf_2dim_LI}

In this section we prove Theorems \ref{thm_2dimLI} and \ref{thm_rigiditylowdim}.  The proof of Theorem \ref{thm_2dimLI} consists of two main steps, each of which is a separate lemma below.

\begin{lemma} \label{lemma_approxcoord}
Let $E$ be a compact null set in $\R^2$ and let $E^1$ and $E^2$ be as in Theorem \ref{thm_disjointcovering}.  For $i \in \{1,2\}$, there exist Lipschitz functions $\{ \varphi_{i,j} \}_{j=1}^\infty$ on $\R^2$ so that
\begin{enumerate}
\item
$\varphi_{i,j}$ converges pointwise to $x_2$;
\item
each $\varphi_{i,j}$ is $3$-Lipschitz and piecewise linear;
\item
for each $p \in E^i$
there is a \emph{closed} neighborhood $K$ containing $p$ so that $\varphi_{i,j}|K$ depends only on the variable $x_i$.
\end{enumerate}
\end{lemma}

To simplify the proof, we divide it into cases of increasing geometric complexity.

\begin{proof}
By Theorems \ref{thm_covering} and \ref{thm_disjointcovering} we have $E = E^1 \cup E^2$, where each $E^i$ has the following properties: given $j \in \N$, there are numbers $N \in \N$ and $\d > 0$ and $x_i$-stripes $\{\n^i(f_{l,j};\d_j)\}_{l=1}^N$ so that

\begin{itemize}
\item[($1'$)] 
$E^i \subset \bigcup_{l=1}^{N_j} \n^i(f_{l,j};\d)$ and $N \cdot \d < 2^{-j}$;
\item[($2'$)] for $l \neq l'$, the interiors of $\n^i(f_{l,j};\d)$ and $\n^i(f_{l',j};\d)$ are disjoint;
\item[($3'$)] each $f_{l,j}$ is a piecewise-linear function, with finitely many points of non-differentiablity.
\end{itemize}

For simplicity, let $i=1$ and $E \subset [0,1]^2$.  For each $l$, put $\n^1_l := \n^1(f_{l,j};\d)$.  To emphasize the dependence on $j$, put $\m_j := \R^2 \setminus \bigcup_l \n^1_l$. Now consider the functions 
$$
\varphi_{1,j}(p) \;:=\; \int_{\{p_1\} \times [0,p_2]} \chi_{\m_j} \,d\H^1
$$
where $p = (p_1,p_2) \in \R^2$.  Property $(1')$ then implies that
$$
0 \;\leq\;
p_2 - \varphi_{1,j}(p_1) \;\leq\;
\sum_{j=1}^N \int_{\{p_1\} \times \R} \chi_{\n^1_l} \,d\H^1 \;=\;
N \cdot \d \;<\; 2^{-j}
$$
from which we obtain Property (1).

\begin{claim} \label{claim_uniflip}
The sequence $\{\varphi_{1,j}\}_{j=1}^\infty$ is uniformly $3$-Lipschitz.
\end{claim}

%To begin, 
Let $p = (p_1,p_2)$ and $q = (q_1,q_2)$ be %arbitrary 
points in $\R^2$.  We argue by cases.

\emph{Case A: $p$ and $q$ lie on the same vertical line}.  %Since $\varphi_{1,j}$ is an indefinite integral of a characteristic function, it
By construction, $\varphi_{1,j}$ is $1$-Lipschitz in the variable $x_2$.  The claim then follows from %the estimate
\begin{equation} \label{eq_sameline}
|\varphi_{1,j}(p) - \varphi_{1,j}(q)| \;\leq\;
|p_2 - q_2| \;\leq\; |p- q|.
\end{equation}

\emph{Case B: $p$ and $q$ lie on the same stripe $\n^1_l$}.
Since $\varphi_{1,j}$ is constant on each of the vertical segments
$\n^1_l \cap (\{p_1\} \times \R)$ and
$\n^1_l \cap (\{q_1\} \times \R)$, the corresponding (lower) endpoints
$p' = (p_1, f_{l,j}(p_1) - \d/2)$ and
$q' = (q_1, f_{l,j}(q_1) - \d/2)$ satisfy
\begin{equation} \label{eq_samesegment}
\varphi_{1,j}(p) \;=\; \varphi_{1,j}(p')
\; \textrm{ and } \;
\varphi_{1,j}(q) \;=\; \varphi_{1,j}(q').
\end{equation}
By Property $(2')$, the interiors of $\{\n^1_l\}_{l=1}^N$ are pairwise disjoint.  A ray with initial point $p'$ and direction $-\vec{e}_2$ crosses through $l-1$ stripes of thickness $\d$, so 
\begin{equation} \label{eq_countstripes}
\varphi_{1,j}(p') \;=\; f_{l,j}(p_1) - (l-1)\d  \,\text{ and }\,
\varphi_{1,j}(q') \;=\; f_{l,j}(q_1) - (l-1)\d.
\end{equation}
Since $L(f_{l,j}) \leq 1$, the claim follows from equations \eqref{eq_samesegment} and \eqref{eq_countstripes}:
\begin{equation} \label{eq_samestripe}
\left.\begin{split}
\hspace{.25in}
|\varphi_{1,j}(p) - \varphi_{1,j}(q)| &=\;
|\varphi_{1,j}(p') - \varphi_{1,j}(q')| \\ &=\;
|f_{l,j}(p_1) - f_{l,j}(q_1)| \;\leq\;
|p_1 - q_1| \;\leq\;
|p - q|.
\hspace{.25in}
\end{split}\right\}
\end{equation}

\emph{Case C: $p, q \notin \m_j$, and both points lie between the same pair of stripes}.
The argument is similar to Case B.  If $p$ and $q$ lie between $\n^1_{l+1}$ and $\n^1_l$, for some $l$, then put $p'' := (p_1, f_{l,j}(p_1) + \d/2)$ and $q'' := (q_1, f_{l,j}(q_1) + \d/2)$.  From the observation
$$
\varphi_{1,j}(p) \;=\; p_2'' - l \cdot \d \; \textrm{ and } \;
\varphi_{1,j}(q) \;=\; q_2'' - l \cdot \d,
$$
we obtain estimates \eqref{eq_countstripes} and \eqref{eq_samestripe} as before.

\emph{Case D: $p$ and $q$ are arbitrary}.
Suppose that $p$ and $q$ are separated by a boundary curve $\a$ of some stripe $\n^1_l$.  Without loss of generality, let
$$
\a \;=\;
\{ (x_1,x_2) \,:\, x_2 = f_{l,j}(x_1) + \d/2 \}
$$
and moreover, assume $p$ lies below $\a$ and $q$ lies above $\a$:
$$
\left\{\begin{split}
p_2' &:=\; f_{l,j}(p_1) + \d/2 \;\geq\; p_2, \\
q_2' &:=\; f_{l,j}(q_1) + \d/2 \;\leq\; q_2.
\end{split}\right.
$$
Note that $p' := (p_1,p_2')$ and $q' := (q_1,q_2')$ lie on the same vertical lines as $p$ and $q$, respectively, and both $p'$ and $q'$ lie on $\a$.  

\begin{figure}[htp]
\centering
\begin{pspicture}(0,0.5)(6,2.7)
%BEFORE BUMP
\psline[](1,1.5)(3,2.5)(5,1.5)
\psline[](1,0.5)(3,1.5)(5,0.5)
%\psline[linestyle=dashed](1,1)(3,2)(5,1)
\put(.7,1.5){$\a$}
%\put(4.5,1){$\n^l_1$}
\psline[linecolor=red](2,2)(2,0.5)
\pscircle[linecolor=red](2,2){.07}
\pscircle[linecolor=red](2,0.5){.07}
\put(2.15,0.5){\textcolor{red}{$p$}}
\put(2.15,1.85){\textcolor{red}{$p'$}}
\psline[linecolor=red](4.5,1.75)(4.5,2.5)
\pscircle[linecolor=red](4.5,1.75){.07}
\pscircle[linecolor=red](4.5,2.5){.07}
\put(4.65,2.5){\textcolor{red}{$q$}}
\put(4.65,1.85){\textcolor{red}{$q'$}}
\end{pspicture}
\caption{A possible configuration for $p$, $q$, $p'$, and $q'$.}
\end{figure}
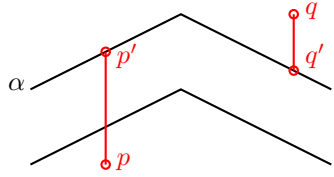

Using the Triangle Inequality and inequalities \eqref{eq_sameline} and \eqref{eq_samestripe}, %we estimate:
\begin{equation} \label{eq_almostuniflip}
\left.\begin{split}
|\varphi_{1,j}(p) - \varphi_{1,j}(q)| \;\leq\; &\;
|\varphi_{1,j}(p) - \varphi_{1,j}(p')| +
|\varphi_{1,j}(p') - \varphi_{1,j}(q')| \\ &\; +
|\varphi_{1,j}(q') - \varphi_{1,j}(q)| \\ \;\leq\; &\;
|p_2 - p_2'| + |f_{l,j}(p_1) - f_{l,j}(q_1)| + |q_2 - q_2'| \\ \;\leq\; &\;
|p_2 - p_2'| + 1 \cdot |p_1 - q_1| + |q_2 - q_2'|.
\end{split}\right\}
\end{equation}

\begin{claim} \label{claim_intervals}
For all choices of $p_2 \leq p_2'$ and $q_2' \leq q_2$, we have
$$
|p_2 - p_2'| + |q_2' - q_2| \;\leq\;
|p_2 - q_2| + |p_2' - q_2'|.
$$
\end{claim}

The argument is combinatorial, so we further proceed by sub-cases.  Consider intervals $I_p := [p_2',p_2]$ and $I_q := [q_2,q_2']$

\emph{Subcase D1: $I_p$ and $I_q$ are disjoint}.
Relative to $p_2 \leq q_2$ or $q_2 \leq p_2$, the union $[p_2,p_2'] \cup [q_2',q_2]$ lies in either $[p_2,q_2]$ or $[q_2',p_2']$.  The claim then follows from
$$
|p_2 - p_2'| + |q_2' - q_2| \;\leq\;
|p_2 - q_2| \vee |p_2' - q_2'| \;\leq\;
|p_2 - q_2| + |p_2' - q_2'|.
$$

\emph{Subcase D2: $I_p \subset I_q$}.
Under this set inclusion, we have the identities
$$
I_q \;=\; [p_2',q_2'] \cup [p_2,q_2], \; \;
I_p \;=\; [p_2',q_2'] \cap [p_2,q_2]
$$
from which we obtain the claim, also as an identity:
\begin{eqnarray*}
|p_2 - p_2'| + |q_2' - q_2| &=&
m_1(I_p) + m_1(I_q) \\ &=&
m_1\big([p_2',q_2'] \cup [q_2,p_2]\big) + m_1\big([p_2',q_2'] \cap [q_2,p_2]\big) \\ &=&
|p_2 - q_2| + |p_2' - q_2'|.
\end{eqnarray*}
By symmetry, the claim also holds for %the sub-case 
$I_q \subset I_p$.
%$[q_2,q_2'] \subset [p_2',p_2]$.

\emph{Subcase D3: $I_p \not\subset I_q$, $I_p \not\subset I_p$, and $I_p \cap I_q \neq \emptyset$}.  %As a final case, suppose that %$[p_2',p_2]$ and $[q_2,q_2']$, $I_p$ and $I_q$ are not disjoint and that neither interval is contained in the other.  This means that o
Of the intervals $[q_2,p_2]$ and $[p_2',q_2']$, one is %the union %$[p_2',p_2] \cup [q_2,q_2']$
$I_p \cup I_q$ and the other is %the intersection %$[p_2',p_2] \cap [q_2,q_2']$.
$I_p \cap I_q$.  We then compute 
\begin{eqnarray*}
|p_2 - p_2'| + |q_2' - q_2| &=&
m_1(I_p) + m_1(I_q) \\ &=&
m_1(I_p \cup I_q) + m_1(I_p \cap I_q) \;=\;
|p_2 - q_2| + |q_2' - p_2'|.
\end{eqnarray*}
This proves Claim \ref{claim_intervals}.  Using this and inequality \eqref{eq_almostuniflip}, 
 Claim \ref{claim_uniflip} follows from
\begin{eqnarray*}
|\varphi_{1,j}(p) - \varphi_{1,j}(q)| &\leq&
|p_2 - p_2'| + |p_1 - q_1| + |q_2 - q_2'| \\ &\leq&
|p_2 - q_2| + |p_1 - q_1| + |p_2' - q_2'| \\ &\leq&
|p_2 - q_2| + 2 \cdot |p_1 - q_1| \;\leq\;
3 \cdot |p - q|.
\end{eqnarray*}
Let $O$ be any connected component of $\m_j$.  From the argument in Case C, the restriction $\varphi_{1,j}|O$ is a translate of $x_2$, so $\varphi_{i,j}|O$ is piecewise linear.  On the other hand, by Property $(3')$, each $f_{l,j}$ is piecewise-linear, so by Case B, the restriction of $\varphi_{i,j}$ to any $x_1$-stripe $\n_j^1$ is also piecewise linear.  Both types of sets $O$ and $\n_j^1$ partition $\R^2$, so $\varphi_{1,j}$ must be piecewise-linear on all of $\R^2$.  This gives Property (2).

Lastly, recall from Case B that for all stripes $\n^1_l$, we have
\begin{equation} \label{eq_univariate}
\varphi_{1,j}(p) \;=\; f_{l,j}(p_1) + (l-1)\d.
\end{equation}
for all $p = (p_1,p_2) \in \n^1_l$.
This gives Property (3): for all $p \in E^1$, there is a \emph{closed} neighborhood $K$ containing $p$ so that $\varphi_{1,j}|K$ depends only on the variable $x_1$.
%\end{rmk}
\end{proof}

\subsection{Linearly Independent Derivations on $\R^2$}

Using the approximating sequence from Lemma \ref{lemma_approxcoord}, we proceed to a linear dependence relation for derivations (with respect to singular measures).

\begin{lemma} \label{lemma_linrel}
Let $\mu$ be a singular Radon measure on $\R^2$, and let $E$ be a subset on which $\mu$ is concentrated.  There exist $F^1, F^2 \subset \R^2$ and $g_1, g_2 \in L^\infty(\R^2,\mu)$ so that $E = F^1 \cup F^2$ and so that, for all $\d \in \U(\R^2,\mu)$,
\begin{equation} \label{eq_linrel}
\left\{
\begin{split}
\d x_2 &=\; %&=&
g_1 \cdot \d x_1 \;\; \mu\textrm{-a.e.\ on } F_1, \\
\d x_1 &=\; %&=&
g_2 \cdot \d x_2 \;\; \mu\textrm{-a.e.\ on } F_2.
\end{split}
\right.
\end{equation}
%hold, .
\end{lemma}

\begin{proof}
We proceed by cases.

Case 1: assume that $E$ is compact, so by Lemma \ref{lemma_approxcoord} there exist piecewise-linear, $3$-Lipschitz functions $\{\varphi_{1,j}\}_{j=1}^\infty$ and $\{\varphi_{2,j}\}_{j=1}^\infty$ that satisfy Properties (1) and (2) of Lemma \ref{lemma_approxcoord}.  Assume again that $i = 1$ and put $\varphi_j := \varphi_{1,j}$.

From the proof of Lemma \ref{lemma_approxcoord}, each $\varphi_j$ is formed from a covering of $E^1$ by $x_1$-stripes $\{\n^1(f_l;\d)\}_{l=1}^N$ with $N \cdot \d < 2^{-j}$.  Moreover, the set of non-differentiability of $\varphi_j$ consists of two parts:

\begin{enumerate}
\item a finite union of $x_1$-stripe boundaries, written $\Gamma := \bigcup_l \partial\n^1(f_l;\d)$;
\item a finite union of vertical line segments, written 
$\ell := \bigcup_k \ell_k$, where each %image set $\proj_{\R \times 0}(\ell_k)$ is
segment $\ell_l$ projects to 
a point of non-differentiability of $f_l$.%, for some $l$.
\end{enumerate}

Let $\d_1, \d_2 \in \U(\R^2,\mu)$ be arbitrary and %.  As a shorthand, 
write $\n^1_l := \n^1(f_l;\d)$.  Several reductions follow, which we state below as claims.

\begin{claim}
Lemma \ref{lemma_linrel} is true for $\mu(\R^2 \setminus \ell) = 0$.
\end{claim}

We may assume that $\mu(\ell) > 0$, % and hence 
so $\chi_\ell \neq 0$.  Since $\H^1(E^1 \cap \ell_k) = 0$ holds for each $k$, it follows from Lemma \ref{lemma_curveintersect} that $E^1 \cap \ell$ is purely $1$-unrectifiable.  Moreover, by Theorem \ref{lemma_pure1unrect}, we have $\chi_\ell\d = 0$, for all $\d \in \U(\R^2,\mu)$, from which we obtain $\chi_\ell(\d_1 + \d_2) = 0$.  This proves the claim.

\begin{claim}
Lemma \ref{lemma_linrel} is true for $\mu(\R^2 \setminus \Gamma) = 0$.
\end{claim}

For each $j \in \N$, let $S_l := \partial\n^1(f_l;\d)$, so $\varphi_j$ is non-differentiable on $S_l$.  In particular, every such $S_l$ is a Lipschitz curve, hence $1$-rectifiable, so by Theorem \ref{lemma_1rect} the module $\U(S_l,\mu)$ has rank-$1$.  This means that the set $\{\chi_{S_l}\d_i\}_{i=1}^2$ is linearly dependent, and by Lemma \ref{lemma_lindepsubset}, so is $\{\chi_\Gamma\d_i\}_{i=1}^2$.  The claim follows.

Without loss of generality, assume that $E^1 = E^1 \setminus (\ell \cup \Gamma)$.  By Theorem \ref{thm_disjointcovering}, $f_l$ is piecewise-linear and $f_l'(\R \setminus \proj_{\R \times 0}(\ell))$ is a finite set in $\R$.  We may then cover $E^1$ by a finite union of sets of the form
$$
\n^1_l(\xi) \;:=\; \{ p \in \n^1_l \,:\, f_l'(p_1) = \xi \}.
$$
Note that the restriction of $\varphi_{1,j}$ to the interior of $\n^1_l(\xi)$ is linear and hence smooth.  From formulas \eqref{eq_chainrule} and \eqref{eq_univariate}, we then obtain the $\mu$-a.e.\ identity
$$
\chi_{\n^1_l(\xi)} \cdot \d \varphi_{1,j} \;=\;
\chi_{\n^1_l(\xi)} \cdot f_l' \cdot \d x_1,
$$
for all $l$ and all $\xi$.  Putting
$G_j^1 := \sum_{l=1}^N \chi_{\n^1_l} \cdot f_l'$,
we further obtain
$$
\d \varphi_{1,j} \;=\; G_j^1 \cdot \d x_1
$$
$\mu$-a.e.\ on $E^1$.  Clearly $\|G_j^1\|_{\infty,\,\mu} \leq 1$ holds for each $j$, and by the Banach-Alaoglu Theorem 
there is a weak-$*$ convergent subsequence $\{G_{j_m}^1\}_{m=1}^\infty$ in $L^\infty(\R^2,\mu)$; let $G^1$ be the weak-$*$ limit.  For all $\d \in \U(\R^2,\mu)$ we also have %the convergence
$$
G_{j_m}^1 \cdot \d x_1 \,\wsto\, G^1 \cdot \d x_1 \textrm{ in } L^\infty(\R^2,\mu).
$$
However, by Property (1) of Lemma \ref{lemma_approxcoord}, we have $\varphi_{1,j} \wsto x_2$ in $\Lip_b(\R^2)$ and by continuity of $\d$, we obtain $\d\varphi_{1,j} \wsto \d x_2$ in $L^\infty(\R^2,\mu)$ for all $\d$.  Putting $F^1 := E^1$ and $g_1 := G^1$, formula \eqref{eq_linrel} follows from uniqueness of limits.

Case 2: For non-compact  $E$, consider subsets in $\R^2$ of the form
$$
Q_{ab} \;:=\; [a,a+1) \times [b,b+1), \; a,b \in \Z.
$$
Indeed, each $Q = Q_{ab}$ is bounded and therefore has finite $\mu$-measure.  From the Borel regularity of $\mu$, there are sequences of compact sets $\{K_c\}_{c=1}^\infty$ so that
$$
\lim_{c \to \infty} \mu\big( (E \cap Q) \setminus K_c \big) \;=\; 0.
$$
Since $K_c$ is compact, there exist subsets $F_c^1$, $F_c^2$ in $\R^2$ so that
$$
K_c \;=\; F_c^1 \cup F_c^2
$$
and there exist functions $G_c^1 \in L^\infty(\R^2,\mu)$ so that the identity
\begin{equation} \label{eq_almostlinrel}
\chi_{K_c} \cdot \d x_2 \;=\; \chi_{K_c} \cdot G_c^1 \cdot \d x_1
\end{equation}
holds $\mu$-a.e.\ on $F_c^1$, for all $\d \in \U(\R^2,\mu)$.  We now put
\begin{eqnarray*}
F^i &:=& \bigcup_{abc} F_c^i, \textrm{ for } i = 1, 2, \\
E_c &:=&
\left\{\begin{array}{ll}
K_1, & c = 1 \\
K_c \setminus K_{c-1}, & c \geq 2,
\end{array}\right.\\
g_1 &:=&
\sum_c \chi_{E_c} \cdot G_c^1.
\end{eqnarray*}
Clearly, $E = \bigcup_c E_c$ and $E = F^1 \cup F^2$.  From formula \eqref{eq_almostlinrel} and from the definitions of $E_c$ and $g_1$, we also obtain formula \eqref{eq_linrel} for $i = 1$.  The lemma follows.
\end{proof}

%\textcolor{red}{Add a little preamble, here?}

\begin{proof}[Proof of Theorem \ref{thm_2dimLI}]
If $\mu$ is a singular Radon measure on $\R^2$, then let $E$ be a null set on which $\mu$ is concentrated.  Let $\d_1, \d_2 \in \U(\R^2,\mu)$ be arbitrary.

By Theorem \ref{lemma_linrel}, there are subsets $F^1$ and $F^2$ so that $E = F^1 \cup F^2$ and there are functions $g_1, g_2 \in L^\infty(\R^2,\mu)$ so that the system of equations \eqref{eq_linrel} holds $\mu$-a.e.\ for $\d_1$ and for $\d_2$.  Now consider 
\begin{equation} \label{eq_2dimscalars}
\left\{\begin{split}
\lambda_1 &:=\;
\chi_{F^1} \cdot \d_2 x_1 \,+\,
\chi_{F^2} \cdot \d_2 x_2, \\
\lambda_2 &:=\;
\chi_{F^1} \cdot \d_1 x_1 \,+\,
\chi_{F^2} \cdot \d_1 x_2.
\end{split}\right.
\end{equation}
We first observe that, for $\mu$-a.e.\ $p \in F^1$, we have the identities
\begin{eqnarray*}
\chi_{F^1} \cdot \big(\lambda_1 \d_1 - \lambda_2 \d_2\big)x_1 &=&
\chi_{F^1} \cdot \big(
\d_2x_1 \cdot \d_1x_1 - \d_1x_1 \cdot \d_2x_1 \big) \;=\; 0, \\
\chi_{F^1} \cdot \big( \lambda_1 \d_1 - \lambda_2 \d_2 \big)x_2 &=&
\chi_{F^1} \cdot \big( \d_2x_1 \cdot \d_1x_2 - \d_1x_1 \cdot \d_2x_2 \big)\\ &=&
\chi_{F^1} \cdot \big(
\d_2x_1 \cdot g_1 \cdot \d_1x_1 -
\d_1x_1 \cdot g_1 \cdot \d_2x_1
\big) \;=\; 0.
\end{eqnarray*}
Arguing similarly for $F^2$, we see that $\lambda_1\d_1 - \lambda_2\d_2$ annihilates both $x_1$ and $x_1$.  By Lemma \ref{lemma_chainrule}, it follows that $\lambda_1\d_1 - \lambda_2\d_2 = 0$.

Now suppose that both $\lambda_1$ and $\lambda_2$ are zero.  By equations \eqref{eq_linrel} and \eqref{eq_2dimscalars}, the four functions $\d_1x_1$, $\d_1x_2$, $\d_2x_1$, and $\d_2x_2$ would all be zero, which implies that $\d_1 = \d_2 = 0$.  This is a contradiction, so either $\lambda_1 \neq 0$ or $\lambda_2 \neq 0$, and therefore the set $\{\d_1,\d_2\}$ is linearly dependent in $\U(\R^2,\mu)$.
\end{proof}

We now prove the rigidity theorem for derivations.

\begin{proof}[Proof of Theorem \ref{thm_rigiditylowdim}]
We argue by contradiction.  For $k=2$, let $\mu$ be a Radon measure on $\R^2$.  If $\mu_S \neq 0$, then let $A$ be a null set on which $\mu_S$ is concentrated.  For any two derivations $\d_1, \d_2$ in $\U(\R^2,\mu)$, the restrictions $\chi_A\d_1, \chi_A\d_2$ lie in $\U(\R^2,\mu_S)$, by Theorem \ref{thm_locality}, and by Theorem \ref{thm_2dimLI}, there exist functions $\lambda_1, \lambda_2 \in \Li(\R^2,\nu_S)$, not both zero, so that
$$
\lambda_1(\chi_A\d_1) \,+\, \lambda_2(\chi_A\d_2) \;=\; 0.
$$
So from the choice of scalars $\Lambda_i = \chi_A\cdot\lambda_i$, $i=1,2$, we see that $\{\d_1,\d_2\}$ is a linearly dependent set in $\U(\R^2,\mu)$.

A similar argument holds for $k=1$, with Lemma \ref{lemma_dim1} in place of Theorem \ref{thm_2dimLI}.
\end{proof}

In the previous section we studied %derivations with respect to 
Radon measures concentrated on $1$-sets.  %Recall that Lemma \ref{lemma_dim1} gives a characterization of derivations on $\R^1$.  
From Lemma \ref{lemma_dim1} we deduced Theorem \ref{thm_1sets}, which asserts that the rank of such 
modules of derivations is at most one.  As an application of Theorem \ref{thm_2dimLI}, we now deduce the following result about derivations on $2$-sets in $\R^n$.

\begin{prop} \label{prop_2sets}
Let $\mu$ be a Radon measure on $\R^n$.
\begin{enumerate}
\item If $\mu$ is concentrated on a $2$-set $A$, then $\U(\R^n,\mu)$ has rank at most $2$.
\item If $A$ contains a purely $2$-unrectifiable subset of positive $\H^2$-measure, then $\U(\R^n,\mu)$ has rank at most $1$.
\end{enumerate}
\end{prop}

\begin{proof}[Sketch of Proof]
By Lemma \ref{lemma_rectdecomp}, we have $A = E \cup F$, where $E$ is $2$-rectifiable and $F$ is purely $2$-unrectifiable.  It is easy to see that derivations restricted to $E$ are pushforwards of derivations on $\R^2$; by Theorem \ref{thm_rigiditylowdim}, the rank of $\U(E,\mu)$ is therefore at most $2$.

For the purely $2$-unrectifiable part, %one uses orthogonal projections onto $2$-dimensional subspaces.  B
by Theorem \ref{thm_projections} the image of $F$ under a generic projection is a null set in $\R^2$.  This produces linear dependence relations betwen derivations as in Lemma \ref{lemma_linrel}.  Arguing similarly as in the proof of Lemma \ref{lemma_pure1unrect}, these linear relations can be ``pulled back'' to $\R^n$.  By choosing scalars $\lambda_i \in L^\infty(\R^n,\mu)$ similarly to those in the proof of Theorem \ref{thm_2dimLI}, we conclude that $\U(F,\mu)$ must have rank at most $1$.
\end{proof}

%\vfill
%\pagebreak
%=========================================================================
\section{Derivations on Spaces Supporting a Poincar\'e Inequality} \label{PI-spaces}

We now turn to the class of metric measure spaces which admit a Poincar\'e inequality in a suitably weak sense.
These were first considered in the work of Heinonen and Koskela in their study of quasiconformal mappings on metric spaces \cite{HeinonenKoskela}, and it is known that such spaces possess good geometric properties, such as quasi-convexity \cite{DS}.  As stated before in the Introduction, Cheeger has also proven an analogue of the Rademacher theorem on such spaces \cite{Cheeger}.  

In what follows, we discuss facts about Sobolev spaces on metric measure spaces that support a $p$-Poincar\'e inequality and then construct derivations on such spaces with respect to the underlying measure.  As an application, we also prove the $2$-dimensional case of Cheeger's conjecture about the structure of such measures.

\subsection{Calculus on Metric Spaces}

As before, $(X,\rho,\mu)$ denotes a metric measure space.  Here and in the remainder of the section we assume that $\mu$ is \emph{doubling}, as defined in Equation \eqref{eq_doubling}.

\begin{rmk} \label{rmk_doubling-separable}
If $X$ admits a doubling measure then the metric on $X$ is also \emph{doubling}, that is: there exists $N \in \N$ so that every ball $B$ in $X$ can be covered by $N$ balls of half the radius of $B$.

By iterating the doubling property above, we see that every ball $B$ in $X$ is a separable metric space.  It follows from Part (2) of Lemma \ref{lemma_ctyderiv} that a linear operator $\d : \Lip_b(B) \to \Li(B,\mu)$ is weak-$*$ continuous on bounded sets if and only if it is sequentially weak-$*$ continuous.
\end{rmk}

%Towards a definition of $p$-PI spaces, 
Following \cite{HeinonenKoskela}, we now introduce the notion of an upper gradient.

\begin{defn} \label{defn_uppergrad}
Let $u : X \to \R$ be Borel.  A Borel function $g : X \to [0,\infty]$ is an \emph{upper gradient} for $u$ if the inequality
\begin{equation} \label{eq_upgradineq}
|u(y) - u(x)| \;\leq\; \int_a^b g(\gamma(t)) \,dt
\end{equation}
holds for all rectifiable curves $\gamma : [a,b] \to X$ which are parametrized by arc-length and which satisfy $x = \gamma(a)$ and $y = \gamma(b)$.
\end{defn}

\begin{defn}
We say that $(X,\rho,\mu)$ supports a \emph{$p$-Poincar\'e inequality} if there exist $\Lambda \geq 1$ and $C > 0$ so that for all balls $B$ in $X$ and all $u \in L^1_{loc}(X,\mu)$, we have
\begin{equation}
\dashint_B |u - u_B| \,d\mu \;\leq\; 
C \cdot \diam(B) \cdot \Big( \dashint_{\Lambda B} g^p \,d\mu \Big)^{1/p}.
\end{equation}
whenever $g$ is an upper gradient of $u$.  As a shorthand, we call $(X,\rho,\mu)$ a \emph{$p$-PI space} if $\mu$ is doubling and if $(X,\rho,\mu)$ admits a $p$-Poincar\'e inequality.
\end{defn}

Following \cite[Sect 2]{Cheeger}, for $u \in L^p(X,\mu)$ we now define
\begin{equation} \label{eq_sobolevnorm}
\|u\|_{1,p} \;:=\; \|u\|_{\mu,p} + \inf_{\{g_i\}} \liminf_{i \to \infty} \|g_i\|_{\mu,p},
\end{equation}
where the infimum is taken over all sequences $\{u_i\}_{i=1}^\infty$ in $L^p(X,\mu)$ so that $u_i \to u$ in $L^p$-norm and so that $g_i$ is a upper gradient for $u_i$, for each $i \in \N$.  

\begin{defn}
The Sobolev space $H^{1,p}(X,\mu)$ is the subspace of functions $u \in L^p(X,\mu)$ for which $\|u\|_{1,p} < \infty$.
\end{defn}

Indeed, $\| \cdot \|_{1,p}$ is a norm on $H^{1,p}(X,\mu)$, but more is true; the next theorem summarizes \cite[Thms 2.7, 2.10, 2.18, 4.48]{Cheeger}.%, \cite[Thms 2.7 \& 4.48]{Cheeger}.

\begin{thm}[Cheeger, 1999] \label{thm_sobolevreflexivity}
$(H^{1,p}(X,\mu), \|\cdot\|_{1,p})$ is a Banach space.  Moreover,
\begin{enumerate}
\item 
if $p > 1$ and if $(X,\rho,\mu)$ is a $p$-PI space, then $H^{1,p}(X,\mu)$ is reflexive;
\item
for each $f \in H^{1,p}(X,\mu)$, there exists $g_f \in L^p(X,\mu)$ so that 
$$
\|f\|_{1,p} \;=\; \|f\|_{\mu,p} + \|g_f\|_{\mu,p}.
$$
If $g$ %\in L^pX,\mu)$ 
is an upper gradient of $f$, then $g_f \leq g$ holds $\mu$-a.e.\ on $X$.
\end{enumerate}
\end{thm}

We call $g_f$ the \emph{minimal (generalized) upper gradient of $f$}.

\begin{rmk}
Shanmugalingam has defined \emph{Newtonian-Sobolev spaces} $N^{1,p}(X,\mu)$ that are isometrically equivalent to the spaces $H^{1,p}(X,\mu)$, for $p \in (1,\infty)$ \cite[Thm 4.10]{Sh}.  Her approach uses the notion of \emph{weak upper gradients}, and the spaces $N^{1,p}(X,\mu)$ are norm completions of functions in $L^p(X,\mu)$ which admit weak upper gradients in $L^p(X,\mu)$.  Moreover, for $p \in (1,\infty)$ we have the equivalence
$$
W^{1,p}(\R^n) \;\cong\;
H^{1,p}(\R^n,m_n) \;\cong\;
N^{1,p}(\R^n,m_n)
$$
For further details, see \cite{Sh}, \cite[Chap 5-6]{HeinonenBook}, and \cite{HeinonenNC}.
\end{rmk}

For $f \in \Lip(X)$, the constant $L(f)$ is always an upper gradient for $f$ but rarely the minimal generalized upper gradient.  We instead consider the upper and lower pointwise Lipschitz constants of $f$, defined as
\begin{equation} \label{eq_biglip} 
\left.\begin{split}
\Lip[f](x) &\;:=\;
\limsup_{y \to x} \frac{|f(x) - f(y)|}{\rho(x,y)}, \\
\hspace{.75in}
\lip[f](x) &\;:=\;
\liminf_{r \to 0} \sup_{\rho(x,y) \leq r} \frac{|f(x) - f(y)|}{r},
\hspace{.75in}
\end{split}\right\}
\end{equation}
respectively.
It is clear from formula \eqref{eq_biglip} %and \eqref{eq_littlelip}
that, for all $x \in X$,
\begin{equation} \label{eq_lipconsts}
\lip[f](x) \;\leq\;
\Lip[f](x) \;\leq\; L(f).
\end{equation}
Semmes has shown that $\Lip[f]$ and $\lip[f]$ are upper gradients of $f$ \cite[Lem 1.20]{Semmes}.  Moreover, for $p$-PI spaces $X$, we have the $\mu$-a.e.\ identities \cite[Thm 6.1]{Cheeger}
\begin{equation} \label{eq_lipminupgrad}
g_f(x) \;=\; \Lip[f](x) \;=\; \lip[f](x).
\end{equation}

\subsection{Derivations from Differentiability}

We now state the Cheeger-Rademacher theorem for $p$-PI spaces.  To fix notation, for $f : X \to \R^k$ and $a \in \R^k$, we write $ a \bullet f := \sum_i a_i f_i$ for their (pointwise) inner product.

\begin{thm}[Cheeger, 1999] \label{thm_cheegerrademacher}
Let $(X,\rho,\mu)$ be a $p$-PI space.  There exists $N \in \N$ and a $\mu$-measurable decomposition $\{X^n\}_{n=1}^\infty$ with the following properties: for each $n \in \N$, there exist $k = k(n) \in \N$, $1 \leq k \leq N$ and $\xi^n \in \Lip(X;\R^k)$ so that
\begin{enumerate}
\item 
There exists $K = K(n) > 0$ so that for all $x \in X^n$,
\begin{equation} \label{eq_Liplowerbd}
K \;\leq\;
\inf\left\{ \Lip[ a \bullet \xi^n ](x) \,:\, a \in \R^k, |a| = 1 \right\}.
\end{equation}

\item
For each $f \in \Lip(B)$, there is a unique map $d^nf : X \to \R^k$, with components in $\Li(X^n,\mu)$, so that for $\mu$-a.e.\ $x \in X^n$,
\begin{equation} \label{eq_diffcond}
\limsup_{y \to x} \left|\frac{f(y) -f(x) - d^nf(x) \bullet \big(\xi^n(y)- \xi^n(x)\big)}{\rho(x,y)}\right| \;=\; 0.
\end{equation}
\end{enumerate}
\end{thm}

Put $\xi^n := (\xi^n_1,\ldots,\xi^n_k)$.  To mimic the terminology of 
manifolds, we refer to $\xi^n_i$ as \emph{(Cheeger) coordinates} on 
$X^n$, to $(\xi^n,X^n)$ as \emph{(Cheeger) coordinate charts} 
on $X$, and to $d^nf$ as the \emph{(Cheeger) differential} of $f$ on 
$X^n$.

\begin{rmk}
Inequality \eqref{eq_Liplowerbd} is a tacit consequence of the proof of \cite[Thm 4.38]{Cheeger} and is used to show $\xi^n_i \in \Li(X^n,\mu)$. In fact, the measurable decomposition is chosen so that it is valid on each $X^n$.

Equation \eqref{eq_diffcond} is a reformulation of Part (iii) of \cite[Thm 4.38]{Cheeger}.  In the notation of \cite{Cheeger},
$$
d^\a f(z) \;=\; \big(b^\a_1(z;f), \cdots, b^\a_k(z;f)\big).
$$
\end{rmk}

On $\R^n$, the coordinate $x_i$ is precisely the Lipschitz function whose gradient is the vector $e_i$.  The next corollary is an analogue of this fact for $p$-PI spaces, and it follows directly from the uniqueness of Cheeger differentials.

\begin{cor} \label{cor_coordvector}
Assuming the hypotheses of Theorem \ref{thm_cheegerrademacher}, let $n \in \N$ and let $1 \leq i \leq k(n)$.  Then $d^n\xi_i^n(x) = e_i$ holds for $\mu$-a.e.\ $x \in X^n$. 
\end{cor}

For a $p$-PI space $(X,\rho,\mu)$, Cheeger and Weaver have shown that $\U(X,\mu)$ is nontrivial \cite[Thm 43]{WeaverED}.  However, their argument is non-constructive, so we will prove a quantitative form of their theorem below.

\begin{thm} \label{thm_diffderiv}
Let $(X,\rho,\mu)$ be a $p$-PI space.  For $f \in \Lip(X)$ and $n \in \N$, let $d^nf : X^n \to \R^k$ be as in Theorem \ref{thm_cheegerrademacher}.  For $1 \leq i \leq k$, the linear operator $\d_i^n : \Lip(X^n) \to \Li(X^n,\mu)$ given by
\begin{equation} \label{eq_PIderiv}
\d_i^nf \;:=\; d^nf \bullet e_i
\end{equation}
is a derivation in $\U(X^n,\mu)$.
\end{thm}

To prove Theorem \ref{thm_diffderiv}, we require two lemmas.  The first is similar to the $\Li$-regularity argument in \cite[p.457]{Cheeger}.

\begin{lemma} \label{lemma_derivlipbd}
Let $(X,\rho,\mu)$ be a $p$-PI space.  For each $n \in \N$, there exists $C = C(n) > 0$ so that for all $f \in \Lip(X)$ and $\mu$-a.e.\ $x \in X^n$, we have
$$
|\d^n_f(x)| \;\leq\; C \cdot \Lip[f](x).
$$
\end{lemma}

\begin{proof}
Fix $x \in X^n$ and put $a_0 = d^nf(x)/|d^nf(x)|$.  By Part (2) of Theorem \ref{thm_cheegerrademacher},
\begin{eqnarray*}
\Lip[a_0 \bullet \xi^n](x) &=&
\frac{1}{|d^nf(x)|} \cdot 
\limsup_{y \to x} \frac{| d^nf(x) \bullet \big( \xi^n(y) - \xi^n(x)\big) |}{\rho(x,y)} \\ &=&
\frac{1}{|d^nf(x)|} \cdot 
\limsup_{y \to x} \frac{|f(y) - f(x)| }{\rho(x,y)} \;=\;
\frac{\Lip[f](x)}{|d^nf(x)|}.
\end{eqnarray*}
By Theorem \ref{thm_cheegerrademacher}, there exists $K = K(n) > 0$ so that for $\mu$-a.e.\ $x \in X^n$, inequality \eqref{eq_Liplowerbd} holds for all $|a| = 1$.  In particular, the vector $a_0$ has norm $1$, so from the above identity, we obtain the lemma with $C = 1/K$.
\end{proof}

\begin{lemma} \label{lemma_weaksobolevconv}
Let $p > 1$, let $n \in \N$, and let $\{f_a\}_{a=1}^\infty$ be a sequence in $\Lip_b(X^n)$ so that $f_a \wsto 0$.  If $B$ is a ball in $X$ so that $\mu(B \cap X^n) > 0$, then $f_a|B \wkto 0$ in $H^{1,p}(B \cap X^n,\mu)$.
\end{lemma}

\begin{proof}
Let $B$ be a ball in $X$ so that $B \cap X^n$ has positive $\mu$-measure.  In what follows, we write $B^n := B \cap X^n$
and for each $a \in \N$, we write $f_a = f_a|B$.

Let $\{f_{a_b}\}_{b=1}^\infty$ be any subsequence of $\{f_a\}_{a=1}^\infty$.  Note that $\{f_{a_b}\}_{b=1}^\infty$ is a bounded set in $H^{1,p}(B^n,\mu)$ because by equations \eqref{eq_lipconsts} and \eqref{eq_lipminupgrad}, we have 
%the estimates
\begin{eqnarray*}
\|f\|_{\mu,\,p} &:=& 
\left[ \int_{B^n} |f|^p \,d\mu \right]^{1/p} \;\leq\; 
\|f\|_\infty \cdot (\diam(B))^{1/p} \\
\|g_f\|_{\mu,\,p} &:=& 
\left[ \int_{B^n} |g_f|^p \,d\mu \right]^{1/p} \;\leq\;
\left[ \int_{B^n} L(f)^p \,d\mu \right]^{1/p} \;=\; 
L(f) \cdot (\diam(B))^{1/p}
\end{eqnarray*}
%hold 
for each $f \in H^{1,p}(X,\mu)$.  Therefore, for $C' := C \cdot (\diam(B))^{1/p}$, we obtain 
%the bound
$$
\|f_{a_b}\|_{H^{1,p}(B^n,\mu)} \;\leq\; C'.
$$

By Theorem \ref{thm_sobolevreflexivity} and weak compactness, there is a further subsequence $h_c := f_{a_{b_c}}$, $c \in \N$, and a function $h \in H^{1,p}(B^n,\mu)$ so that $h_c \wkto h$ in $H^{1,p}(B^n,\mu)$.  We now invoke Mazur's Lemma, so there is a sequence of (finite) convex combinations $\tilde{h}_c := \sum_\a \lambda_{c\a} \cdot h_\a$ which converge in norm to $h$ in $H^{1,p}(B^n,\mu)$.  In particular, $\tilde{h}_c$ converges in norm to $h$ in $L^p(B^n,\mu)$, so there is a further subsequence $\{\tilde{h}_{c_d}\}_{d=1}^\infty$ that converges $\mu$-a.e.\ to $h$ on $B^n$.

By hypothesis, $f_a \wsto 0$ in $\Lip_b(X^n)$.  Since $\|f_a\|_{\Lip} \leq C$, it follows from Lemma \ref{lemma_weakstarlip} that $f_a$ converges pointwise to $0$, and therefore $h_c$ also converges pointwise to $0$.  A sharper form of Mazur's Lemma\footnote{This fact follows from applying the usual form of Mazur's Lemma to each of the sequences $\{f_a\}_{a=\a}^\infty$, for $\a \in \N$, and then taking an appropriate ``diagonal'' subsequence.} also assures that $\tilde{h}_c$ converges pointwise to $0$, and therefore $\tilde{h}_{c_d}$ also converges pointwise to $0$.  This shows that $h=0$ $\mu$-a.e.\ on $B^n$, so every subsequence of $\{f_a\}_{a=1}^\infty$ has a further subsequence which converges weakly to $0$ in $H^{1,p}(B^n,\mu)$.  It follows that $f_a \wkto 0 $ in $H^{1,p}(B,\mu)$.
\end{proof}

\begin{proof}[Proof of Theorem \ref{thm_diffderiv}]
Let $n, k \in \N$ be as given in Theorem \ref{thm_cheegerrademacher}, and let $\d^n_i : \Lip_b(X^n) \to \Li(X^n,\mu)$ be the map from formula \eqref{eq_PIderiv}.  

By the uniqueness of Cheeger differentials, the map $f \mapsto d^nf$ is linear, so each $\d^n_i$ is linear.  It is known that $d^n$ satisfies the Leibniz rule \cite[Eqn 4.43]{Cheeger}, and by a similar argument as above, $\d^n_i$ also satisfies the Leibniz rule.
It remains to show that $\d^n_i$ is continuous.  By Lemma \ref{lemma_ctyderiv} and Remark \ref{rmk_doubling-separable}, it suffices to check weak-$*$ convergent sequences in $\Lip_b(X^n)$. 

To this end, let $\{f_a\}_{a=1}^\infty \subset \Lip_b(X^n)$ satisfy $f_a \wsto 0$ and $\sup_\a \|f_a\|_{\Lip} \leq C$, for some $C \in (0,\infty)$.  Fix $p \in (1,\infty)$, and let $q = p/(p-1)$.  As a shorthand, we suppress the notation $d\mu$ below.

Let $\psi \in L^1(X,\mu)$ be given, and fix $\e > 0$ and $x_0 \in X^n$.  Since $\int_X |\psi|$ is finite, there exists $R > 0$ so that 
\begin{equation} \label{eq_est1}
\int_{X \setminus B(x_0,R)} |\psi| \;\leq\; \frac{\e}{3C'}.
\end{equation}
Put $B = B(x_0,R)$.  Since $\mu(B) < \infty$, $L^q(B,\mu)$ is a dense subset of $L^1(B,\mu)$, so there exists $\varphi \in L^q(B,\mu)$ so that
\begin{equation} \label{eq_est2}
\int_B|\psi - \varphi| \;\leq\; \frac{\e}{3C'}.
\end{equation}
Now consider the linear operator given by
$$
T_\varphi(f) \;:=\; \int_{B \cap X^n} \varphi \cdot \d^n_if.
$$  
Put $C'' := C(n) \cdot \|\varphi | B\|_{\mu,\,q}$.  From formula \eqref{eq_lipminupgrad} and Lemma \ref{lemma_derivlipbd} we obtain
$$
|T_\varphi(f)| \;\leq\;
\int_B |\varphi| \cdot |\d^n_if| \;\leq\;
C(n) \cdot \int_B |\varphi| \cdot g_f \;\leq\;
C'' \cdot \|f\|_{H^{1,p}(B,\mu)}
$$
for all $f \in \Lip_b(X^n)$, so $T_\varphi$ is a bounded linear functional on $\Lip_b(X^n) \cap H^{1,p}(B,\mu)$.  By the Hahn-Banach theorem, it extends to an element in $[H^{1,p}(B \cap X^n,\mu)]^*$, which we also call $T_\varphi$.  

From our hypothesis we have $f_a \wsto 0$ in $\Lip_b(X^n)$, so by Lemma \ref{lemma_weaksobolevconv}, we obtain $f_a \wkto 0$ in $H^{1,p}(B,\mu)$.  This implies that, for sufficiently large $a \in \N$,
\begin{equation} \label{eq_est3}
|T_\varphi(f_a)| \;=\; 
\left| \int_B \varphi \cdot \d^n_if_a \right| \;\leq\; 
\frac{\e}{3}.
\end{equation}
We now combine estimates \eqref{eq_est1} through \eqref{eq_est3} to obtain
\begin{eqnarray*}
\left| \int_X \psi \cdot \d^n_if_a \right| &\leq&
\left| \int_{X \setminus B} \psi \cdot \d^n_if_a \right| \,+\,
\left| \int_B (\psi - \varphi) \cdot \d^n_if_a \right| \,+\,
\left| \int_X \varphi \cdot \d^n_if_a \right| \\ &\leq&
C' \cdot \int_{X \setminus B} |\psi| \,+\,
C' \cdot \int_B |\psi - \varphi| \,+\,
\left|\int_B \varphi \cdot \d^n_if_a\right| \\ &\leq&
C' \cdot \frac{\e}{3C'} + C' \cdot \frac{\e}{3C'} + \frac{\e}{3} \;=\; \e.
\end{eqnarray*}
Since the above estimates hold for all $\e > 0$, we obtain $\int_X \psi \cdot \d^n_if_a \to 0$.  However, $\psi \in L^1(X,\mu)$ was also arbitrary, which implies $\d^n_if_a \wsto 0$ in $\Li(X^n,\mu)$.
\end{proof}

\subsection{Geometric Rigidity and Cheeger's Conjecture}

As discussed in the introduction, Theorem \ref{thm_cheegerrademacher} 
indicates that $p$-PI spaces have good infinitesmal geometry, in the  
sense of a differentiability property for Lipschitz functions.

Regarding the {\em global} geometric structure of such spaces, the 
following result was proven by Cheeger \cite[Thm 14.2]{Cheeger}.  
In addition to the hypotheses for differentiability, as in Theorem 
\ref{thm_cheegerrademacher}, one further requires that the coordinate 
charts remain nondegenerate in a measure-theoretic way.

\begin{thm}[Cheeger, 1999] \label{thm_nonembed}
Let $(X,d)$ be a complete metric space that supports a doubling measure 
and a $p$-Poincar\'e inequality, for some $1 < p < \infty$.  Assume 
in addition that $X$ admits an isometric embedding $\iota : X \to \R^N$ 
for some $N \in \N$.  If $\H^{k(n)}(X^n) > 0$, then $\iota(X^n)$ is 
$k(n)$-rectifiable.
\end{thm}

In light of this theorem, Cheeger has conjectured that the images of 
coordinate charts are always measurably non-degenerate \cite[Conj 
4.63]{Cheeger}.  

\begin{conj}[Cheeger, 1999] \label{conj_cheegermeas}
Let $(X,\rho,\mu)$, $\{X_n\}_{n=1}^\infty$ and $\xi^n : X \to \R^{k(n)}$
be as in Theorem \ref{thm_cheegerrademacher}.  Then
$\H^{k(n)}(\xi^n(X^n)) > 0$.
\end{conj}

Since Lipschitz maps do not increase Hausdorff dimension, i.e.\
$$
\H^{k(n)}(\xi^n(X^n)) \;\leq\; L(\xi^n) \, \H^{k(n)}(X^n)
$$
the validity of Conjecture \ref{conj_cheegermeas} is consistent with
the hypothesis of Theorem \ref{thm_nonembed}.

Several special cases of Conjecture \ref{conj_cheegermeas} are known.  
Cheeger has proven it under the hypothesis that $\mu$ is \emph{lower
Ahlfors regular} \cite[Thm 13.12]{Cheeger} with exponent $k(n)$; that
is, there exists $C \geq 1$ so that
$$
C^{-1} r^{k(n)} \;\leq\; \mu(B(x,r))
$$
holds, for all $x \in X$ and all $r > 0$.
Keith has also proven the conjecture for the case $k = 1$, but without 
additional hypotheses \cite{Keith}.

Using results from Section \ref{sect_2dim}, we now prove the conjecture 
for $k=2$ and without additional hypotheses.  To begin, we prove a 
lower bound for the rank of $\U(X,\mu)$.

\begin{lemma} \label{lemma_PIderivsLI}
Let $(X,\rho,\mu)$ be a $p$-PI space.  If $\{\d_i^n\}_{i=1}^k$ are the derivations from formula \eqref{eq_PIderiv}, then they form a linearly independent set in $\U(X^n,\mu)$.
\end{lemma}

\begin{proof}%[Proof of Corollary \ref{cor_PIderivsLI}]
We argue by contradiction.  Suppose there exist $\{\lambda_i\}_{i=1}^k$ in  $\Li(X^n,\mu)$, not all zero, so that $\d' := \sum_i \lambda_i \d_i^n$ is zero.
As a result, $\chi_B \cdot \d'g = 0$, for all $g \in \Lip(X)$ and for all balls $B$ which meet $X^n$.

In particular, let $g = \xi^n_i$.  From Corollary \ref{cor_coordvector}, it follows that $\d_i^n\xi^n_i = 1$ and $\d^n_i\xi^n_j = 0$ whenever $i \neq j$.  Computing further, %we obtain
$$
0 \;=\; 
\chi_B \cdot \d'\xi^n_j \;=\; 
\chi_B \cdot \sum_{i=1}^k \lambda_i \cdot \d^n_i\xi^n_j \;=\; \chi_B \cdot \lambda_i.
$$
So each $\lambda_i$ is zero on every ball $B$, and $\lambda_i = 0$ holds $\mu$-a.e.\ on $X^n$.
\end{proof}

\begin{lemma} \label{lemma_PIpushfwdLI}
Let $(X,\rho,\mu)$ be a $p$-PI space, and let $\{X_n\}_{n=1}^\infty$ and $\xi^n : X \to \R^k$ be as in Theorem \ref{thm_cheegerrademacher}.  Then $\U(\R^k,\xi^n_\#(\mu \lfloor X^n))$ has rank at least $k$.
%Then for $\nu_n := \xi^n_\#(\mu \lfloor X^n)$, the module $\U(\R^k,\nu_n)$ has rank at least $k$.
\end{lemma}

\begin{proof}
By Lemma \ref{lemma_PIderivsLI}, the measure $\mu \lfloor X^n$ admits a linearly independent set of $k$ derivations on $X^n$.  We now claim that for $1 \leq i \leq k$, the pushforward derivations $\d_i' := \xi^n_\#\d^n_i$ form a linearly independent set in $\U(\R^k,\nu_n)$, where $\nu_n := \xi^n_\#(\mu \lfloor X^n)$.

Let $\{\lambda_i\}_{i=1}^k$ be functions in $\Li(\R^k,\nu_n)$ so that $\d' := \sum_i \lambda_i \d_i'$ is zero.  This 
implies that for all $f \in \Lip(\R^k)$ and all balls $B$ in $\R^k$, we have $\chi_B \cdot \d' f = 0$.

In particular, put $f = x_j$ and observe that for $i \neq j$, we have $\d_i'x_j = 0$. %holds whenever $i \neq j$.  
Indeed, it follows from Lemma \ref{lemma_pushfwd}, Corollary \eqref{cor_coordvector}, and formula \eqref{eq_PIderiv} that for each $h \in L^1(\R^k,\nu_n)$ and each ball $B$ in $\R^k$,
$$%\begin{eqnarray*}
\int_B h \cdot \d_i'x_j \,d\nu_n \;=\; %&=&
%\int_{(\xi^n)^{-1}(B)} (h \circ \xi^n) \cdot \d_i^n(x_j \circ \xi^n) \,d\mu  \\ &=&
\int_{(\xi^n)^{-1}(B)} (h \circ \xi^n) \cdot \d_i^n\xi^n_j \,d\mu \;=\; 0.
$$%\end{eqnarray*}
Next, put $Z_n := \{\d_i'x_i = 0 \}$ and  $h := \chi_{Z_n}$. %has $\nu_n$-measure zero, and this is proven in a similar way.  If $Z_n$ were to have positive $\mu$-measure, then putting $h := \chi_{Z_n}$, we compute
A similar computation gives
\begin{eqnarray*}
0 \;=\;
\int_B h \cdot \d_i'x_i \,d\nu_n &=&
\int_{(\xi^n)^{-1}(B)} \chi_{(\xi^n)^{-1}(Z_n)} \cdot \d^n_i\xi^n_i \,d\mu \\ &=&
\mu\big((\xi^n)^{-1}(B \cap Z_n)\big) \;=\;
\xi^n_\#\mu(B \cap Z_n).
\end{eqnarray*}
Letting $B = B(0,m)$, for $m \in \N$, we obtain $\xi^n_\#\mu(Z_n) = 0$ and therefore $\d_i'x_i(x) \neq 0$ for $\xi^n_\#\mu$-a.e.\ $x \in \R^k$.  From these observations, we conclude that
$$
0 \;=\; 
\d'x_j \;=\; 
\sum_{i=1}^k \lambda_i \cdot \d_i'x_j \;=\;
\lambda_j \cdot \d_j'x_j
$$
holds, for each $1 \leq j \leq k$, and therefore $\lambda_j = 0$.  This proves the lemma.
\end{proof}

\begin{cor}
If $k = k(n) = 2$, then $\xi^n_\#(\mu \lfloor X^n)$ is a nonzero measure 
on $\R^2$ that is absolutely continuous to Lebesgue $2$-measure.  In 
particular, Conjecture \ref{conj_cheegermeas} and Theorem 
\ref{thm_bilipembed} are true for $k = 2$.
\end{cor}

\begin{proof}
Put $\nu_n := \xi^n_\#(\mu \lfloor X^n)$.
By Lemma \ref{lemma_PIpushfwdLI}, the module $\U(\R^k,\nu_n)$ has rank at 
least $k$, so $\nu_n$ must be a nonzero measure.  For $k = 2$ this 
implies that $\nu_n$ is absolutely continuous to $m_2$; supposing 
otherwise, if $\nu_n$ had nonzero Lebesgue singular part, then by Theorem 
\ref{thm_2dimLI}, the rank of $\U(\R^2,\nu_n)$ would be at most $1$.

By the definition of pushforward measure, $\nu_n$ is concentrated on the image set $\xi^n(X^n)$.  Because $\nu_n$ is nonzero, we see that $\nu_n(\xi^n(X^n)) > 0$, and because $\nu_n$ is absolutely continuous to $m_2$, we further obtain $m_2(\xi^n(X^n)) > 0$.
\end{proof}

Lastly, we prove the rigidity theorem for doubling measures on $\R^2$ that 
support a Poincar\'e inequality (as formulated in the Introduction).  We 
begin with a lemma.

\begin{lemma} \label{lemma_dimrank}
Assuming the hypotheses of Theorem \ref{thm_euclabsocty}, let 
$\{X^n\}_{n=1}^\infty$ be the measurable decomposition of $X = \R^2$ from 
Theorem \ref{thm_cheegerrademacher}.  Then $\U(X^n,\mu)$ has rank at most 
$2$, for all $n \in \N$.
\end{lemma}

\begin{proof}
Suppose $\{\d_i\}_{i=1}^3$ is linearly independent in $\U(X,\mu)$.  Then 
the derivations
\begin{equation} \label{eq_ONderivs}
\left.\begin{split}
\d_1' &\;:=\; (\d_1x_2)\d_2 - (\d_2x_2)\d_1 \\
\hspace{1.25in}
\d_2' &\;:=\; (\d_2x_1)\d_1 - (\d_1x_1)\d_2 
\hspace{1.25in}
\end{split}\right\}
\end{equation}
satisfy $\d_i'x_j = 0$ $\mu$-a.e.\ on $\R^2$, whenever $i \neq j$.  
Moreover, we have the $\mu$-a.e.\ identity
$$
\d_1'x_1 \;=\; \d_2'x_2 \;\neq\; 0
$$ 
on $\R^2$, otherwise the Chain Rule would imply that $\d_i' = 0$, so 
$\{\d_i\}_{i=1}^3$ would be linearly dependent.  On the other hand, 
equation \eqref{eq_ONderivs} implies that the derivation
\begin{eqnarray*}
\d_3' &:=&
(\d_1'x_1)\d_3 - (\d_3x_1)\d_1' - (\d_3x_2)\d_2' \\ &=&
(\d_1'x_1)\d_3 - 
(\d_3x_1)\big[ (\d_1x_2)\d_2 - (\d_2x_2)\d_1 \big] -
(\d_3x_2)\big[ (\d_2x_1)\d_1 - (\d_1x_1)\d_2 \big] \\ &=&
(\d_1'x_1)\d_3 
+ \det\left[\begin{array}{cc}
\d_1x_1 & \d_1x_2 \\
\d_3x_1 & \d_3x_2
\end{array}\right]\d_2
+ \det\left[\begin{array}{cc}
\d_2x_2 & \d_2x_1 \\
\d_3x_2 & \d_3x_1
\end{array}\right]\d_1
\end{eqnarray*}
acts as zero on every polynomial in $\R^3$ and hence, on every $f \in 
\Lip_b(\R^3)$.  This therefore contradicts the linear independence of
$\{\d_i\}_{i=1}^3$, since $\d_1'x_1$ must be nonzero in 
$L^\infty(\R^3,\mu)$.
\end{proof}

\begin{proof}[Proof of Theorem \ref{thm_euclabsocty}]
Assuming all of the hypotheses, let $\{X^n\}_{n=1}^\infty$ again denote 
the measurable decomposition of $X = \R^2$ from Theorem 
\ref{thm_cheegerrademacher}.  Moreover, a theorem of Keith \cite[Thm 
2.7]{KeithCoord} states that on each chart $X^n$, coordinates can be 
chosen to be ``distance vectors'' --- that is, there exist points $\{ p_i 
\}_{i=1}^{k(n)}$ on $X$ so that 
$$
\xi^n(x) \;:=\; \big( d(x,p_1), \ldots, d(x,p_{k(n)}) \big)
$$
satisfies the differentiability property \eqref{eq_diffcond}.

Applying Theorem \ref{thm_diffderiv} and Lemma \ref{lemma_dimrank}, each 
$\U(X^n,\mu)$ must either have rank-$1$ or rank-$2$.  If $\U(X^n,\mu)$ has 
rank-$1$, for some $n \in \N$, then fix a coordinate function $\xi^n(x) = 
|x - p|$ on $X^n$.  Since the support of $\mu$ is dense in $\R^2$, it 
follows that for every point of density $z = (z_1,z_2)$ of $\mu$ in 
$\R^2$, there are sequences $\{y^i\}_{i=1}^\infty$ and 
$\{z^i\}_{i=1}^\infty$, both converging to $z$,
and 
$$
\frac{y_2^i - z_2}{|y^i - z|} \to 1 \; \text{ and } \;
\frac{z_2^i - z_2}{|z^i-z|} \to 0, \; \text{ as } \; i 
\to \infty.
$$
Applying Theorem \ref{thm_cheegerrademacher} to $f(x) = x_2$, it follows 
that
\begin{eqnarray*}
-d^nf(z) &=& 
\lim_{i \to \infty} \frac{z_2^i - z_2 - d^nf(z)|z^i-z|}{|z^i-z|} \\ &=&
0 \;=\;
\lim_{i \to \infty} \frac{y^i_2 - z_2 - d^nf(z)|y^i-z|}{|y^i-z|} \;=\;
1 - d^nf(z)
\end{eqnarray*}
which is a contradiction.  Therefore each $\U(X^n,\mu)$ must have 
rank-$2$, so the desired conclusion follows from Theorem 
\ref{thm_rigiditylowdim}.
\end{proof}

\begin{rmk}[Open Problems]
It is interesting to note that Lemma \ref{lemma_PIpushfwdLI} holds true 
for all $k \in \N$.  We would obtain all cases of Cheeger's measure 
conjecture if an analogue of Theorem \ref{thm_2dimLI} were also true for 
all $k \in \N$.  Recalling further, Theorem \ref{thm_2dimLI} relies 
crucially on Theorem \ref{thm_disjointcovering}, which is an adaptation 
of the covering theorem of Alberti, Cs\"ornyei, and Preiss (Theorem 
\ref{thm_covering}).  

There are other covering theorems for null sets in $\R^k$, for all $k \in 
\N$ \cite[Prop 8.4]{ACP}.  However, for $k \geq 3$, such covers consist 
of neighborhoods of both $1$-dimensional curves and $(k-1)$-dimensional 
hypersurfaces in $\R^k$.  It is easy to see that the argument in Section 
\ref{sect_2dim} generalizes for the ``$(k-1)$-dimensional'' part of a 
$m_k$-null set, but not the ``$1$-dimensional'' part.

It remains an open question if, for $k \geq 3$, every $m_k$-null set in 
$\R^k$ can be covered by countably many neighborhoods of hypersurfaces, 
each of thickness $r_j$, so that the total thickness $\sum_j r_j$ is 
arbitrarily small \cite[Ques 8.5]{ACP}.  An affirmative answer to this 
question would prove Cheeger's measure conjecture in its full generality.
\end{rmk}

% USED FOR SIMPLICITY OF VIEWING
%\vfill
%\pagebreak
%=========================================================================

\bibliographystyle{alpha}
\bibliography{derivs-paper}

\end{document}